\documentclass[11pt,reqno,tbtags]{amsart}

\usepackage{amsthm,amssymb,amsfonts,amsmath}
\usepackage{amscd} 
\usepackage[numbers, sort&compress]{natbib} 
\usepackage{subfigure, graphicx}
\usepackage{paralist}
\usepackage{hyperref}
\usepackage{color}
\usepackage[normalem]{ulem}
\usepackage{stmaryrd} 
\usepackage[refpage,noprefix,intoc]{nomencl} 

\providecommand{\R}{}
\providecommand{\Z}{}
\providecommand{\N}{}

\renewcommand{\R}{\mathbb{R}}
\renewcommand{\Z}{\mathbb{Z}}
\renewcommand{\N}{{\mathbb N}}

\newcommand{\E}[1]{{\mathbf E}\left[#1\right]}

\newcommand{\va}{{\mathbf{Var}}}
\newcommand{\p}[1]{{\mathbf P}\left\{#1\right\}}

\newcommand{\I}[1]{{\mathbf 1}_{[#1]}}
\newcommand{\set}[1]{\left\{ #1 \right\}}
\newcommand{\Cprob}[2]{\mathbf{P}\set{\left. #1 \; \right| \; #2}} 
\newcommand{\probC}[2]{\mathbf{P}\set{#1 \; \left|  \; #2 \right. }}
\newcommand{\Cexp}[2]{\mathbf{E}\set{\left. #1 \; \right| \; #2}} 
\newcommand{\Cvar}[2]{\mathbf{Var}\set{\left. #1 \; \right| \; #2}} 

\newcommand{\phat}[1]{\ensuremath{\hat{\mathbf P}}\left\{#1\right\}}
\newcommand{\Cprobhat}[2]{\ensuremath{\hat{\mathbf{P}}\set{\left. #1 \; \right| \; #2}}}

\newcommand\cF{\mathcal F}

\newcommand\cP{\mathcal P}

\newcommand\cS{{\mathcal S}}
\newcommand\cT{{\mathcal T}}
\newcommand\cU{{\mathcal U}}

\newcommand{\be}{\mathbf{e}}

\newcommand{\bV}{\mathbf{V}}

\newcommand{\convdist}{\ensuremath{\stackrel{\mathrm{d}}{\rightarrow}}}

\newcommand{\pran}[1]{\left(#1\right)}
\providecommand{\eps}{}
\renewcommand{\eps}{\varepsilon}
\providecommand{\ora}[1]{}
\renewcommand{\ora}[1]{\overrightarrow{#1}}
\newcommand\urladdrx[1]{{\urladdr{\def~{{\tiny$\sim$}}#1}}} 
\DeclareRobustCommand{\SkipTocEntry}[5]{} 

\newtheorem{thm}{Theorem}
\newtheorem{lem}[thm]{Lemma}
\newtheorem{prop}[thm]{Proposition}
\newtheorem{cor}[thm]{Corollary}

\newtheorem{fact}[thm]{Fact}

\numberwithin{equation}{section}
\numberwithin{thm}{section}

\makenomenclature										
\usepackage{setspace}
\newcommand{\rw}{\mathrm{w}}
\newcommand{\rK}{\mathrm{K}}

\newcommand{\rU}{\mathrm{U}}
\newcommand{\core}{\mathop{\mathrm{core}}}
\newcommand{\lr}[1]{\stackrel{#1}{\longleftrightarrow}}
\newcommand{\hull}[2]{#1\langle #2 \rangle}
\newcommand{\dist}{\mathop{\mathrm{dist}}}
\begin{document}

\title{Multi-source invasion percolation on the complete graph} 
\author{Louigi Addario-Berry}
\author{Jordan Barrett}
\address{Department of Mathematics and Statistics, McGill University, Montr\'eal, Canada}
\email{louigi.addario@mcgill.ca}
\urladdrx{http://problab.ca/louigi/}
\email{jordan.barrett@mail.mcgill.ca}
\urladdrx{https://www.math.mcgill.ca/jbarrett/}
\date{August 12, 2022} 

\subjclass[2010]{Primary: 60K35; Secondary: 60C05, 05C80, 82B43, 82C43} 

\begin{abstract} 
We consider invasion percolation on the randomly-weighted complete graph $K_n$, started from some number $k(n)$ of distinct source vertices. The outcome of the process is a forest consisting of $k(n)$ trees, each containing exactly one source. Let $M_n$ be the size of the largest tree in this forest. Logan, Molloy and Pralat \cite{logan18variant} proved that if $k(n)/n^{1/3} \to 0$ then $M_n/n \to 1$ in probability. In this paper we prove a complementary result: if $k(n)/n^{1/3} \to \infty$ then $M_n/n \to 0$ in probability. This establishes the existence of a phase transition in the structure of the invasion percolation forest around $k(n) \asymp n^{1/3}$. 

Our arguments rely on the connection between invasion percolation and critical percolation, and on a coupling between multi-source invasion percolation with differently-sized source sets. A substantial part of the proof is devoted to showing that, with high probability, a certain fragmentation process on large random binary trees leaves no components of macroscopic size. 
\end{abstract}

\maketitle



\section{\bf Introduction}

Fix a locally finite weighted graph $G=(v(G),e(G),\rw)$ such that $\rw:e(G)\to (0,\infty)$ is injective. The {\em invasion percolation} process on $G$ works as follows. 
\begin{itemize}
\item Fix a finite starting set $\cS \subseteq v(G)$, and let $\cS_0=\cS$. 
\item For $1 \le i < |v(G)|+1-|\cS|$, let $e_i\in E$ be the smallest-weight edge from $\cS_{i-1}$ to the rest of the graph. That is, $e_i=uv$ minimizes 
\[
\{\rw_e: e=uv, u \in \cS_{i-1},v \not \in \cS_{i-1}\}.
\]
\item Let $v_i=v$, and set $\cS_i=\cS_{i-1}\cup \{v_i\}$. 
\end{itemize}
Write $F(G,\cS)=(v(F(G,\cS)),e(F(G,\cS)))$ for the subgraph of $G$ with vertex set $\cS\cup \{v_i,0 \le i < |v(G)|+1-|\cS|\}$ and edge set $\{e_i,1 \le i < |v(G)+1-|\cS|\}$. 
Since each edge added by invasion percolation connects to a vertex not incident to any previous edge, the result $F(G,\cS)$ of the invasion percolation is a forest with $|\cS|$ connected components, in which each of the elements of $\cS$ lies in a distinct connected component of $F_{\cS}$. 

Invasion percolation was introduced in \cite{chandler_koplik_lerman_willemsen_1982}, and independently (with a slightly different formulation, using vertex rather than edge weights) in \cite{MR725616}. The latter paper, which coined the term ``invasion percolation'', considered the process on 2- and 3-dimensional lattice rectangles, with the starting set given by the vertices of one boundary side (or boundary face), and with independent random Uniform$[0,1]$ weights. 

The behaviour of invasion percolation with random weights is known to be closely linked to that of {\em critical percolation} on the corresponding graph, and indeed, invasion percolation is one of the simplest examples of self-organized criticality in random systems \cite{Stark:1991aa}. 

Invasion percolation has been extensively studied in the probability and statistical physics communities: on lattices \cite{MR1372437,MR1316507,MR2962082,MR2350578,MR3857861}, on trees \cite{MR708864,MR2393988,MR3059198,MR2977982,MR3940761}, and in the mean-field or general graph setting \cite{MR1340039,MR3857861,MR2023650,ablocal,MR1611522}.  However, past work has almost exclusively focussed on invasion percolation run from a single starting vertex. 

The purpose of this paper is to study mean-field invasion percolation run from starting sets of variable sizes. We establish a phase transition in the structure of the resulting forest, depending on the size of the starting set. 
Write $\rK_n=([n],{[n] \choose 2},\rU)$ for the randomly-weighted complete graph, 
with vertex set $[n]:=\{1,\ldots,n\}$, edge set ${[n] \choose 2}:=\{e \subset [n]: |e|=2\}$, and independent Uniform$[0,1]$ edge weights $\rU=\{U(e),e \in {{[n] \choose 2}}\}$. 
\begin{thm}\label{thm:main}
Fix positive integers $(k(n),n \ge 1)$, and for $n \ge 1$ let $M_n$ 
be the size of the largest connected component of $F(\rK_n,[k(n)])$. 
\begin{itemize}
\item If $k(n)/n^{1/3} \to 0$ then $M_n/n \to 1$ in probability. 
\item If $k(n)/n^{1/3} \to \infty$ then $M_n/n \to 0$ in probability. 
\end{itemize}
\end{thm}

\noindent {\bf Remarks.}

\smallskip
\noindent {$\star$} By the symmetries of the model, the starting set $[k(n)]$ could be replaced by any other set $\cS(n)$ of size $k(n)$ and the same result would hold. 

\smallskip
\noindent {$\star$} The first assertion of the theorem, that if $k(n)/n^{1/3} \to 0$ then $M_n/n \to 1$ in probability, was proved in \cite{logan18variant}. That work also proved that $M_n/n \to 0$ in probability provided that $k(n)/(n^{1/3}(\log n)^{4/3}(\log\log n)^{1/3}) \to \infty$, and provided more quantitative upper bounds on $M_n$ for such values of $k(n)$. Thus, the main contribution of this work is to pin down the location of the phase transition in the behaviour of $M_n$ to $k(n)$ of order precisely $n^{1/3}$. 

\smallskip
\noindent {$\star$} We conjecture that if $k(n)/n^{1/3} \to c \in \R$ then $M_n/n$ converges in distribution to a non-degenerate limit $M_\infty(c)$. More strongly, we make the following conjecture. Write $L_{n,i}$ for the size of the $i$'th largest connected component of $F(\rK_n,[k(n)])$, with $L_{n,i}=0$ if $F(\rK_n,[k(n)])$ has fewer than $i$ connected components. 
For each $c \in \R$ there exists a random vector $(L_{\infty,i}(c),i \ge 1)$ taking values in the set $\Delta_\infty^{\downarrow}=\{(\ell_i,i \ge 1) \in (0,1)^\N: \sum_{i \ge 1} \ell_i=1\}$, such that 
if $k(n)/n^{1/3}\to c$ then $(L_{n,i}/n,i \ge 1) \convdist (L_{\infty,i}(c))$ in the sense of finite-dimensional distributions. 

\subsection{Overview of the rest of the paper}
In Section~\ref{sec:sketch}, we explain several useful connections between invasion percolation, critical percolation, and minimum spanning trees. We then use these connections to prove Theorem~\ref{thm:main}, modulo a key input to the proof. This key input, Proposition~\ref{prop:critical_bound}, roughly states the following. In the case that $k(n)/n^{1/3} \to \infty$, if we run the multi-source invasion percolation process for $n/2+O(n^{2/3})$ steps, then the size of the largest tree is with high probability much smaller than the size of the largest component in an Erd\H{o}s--R\'enyi random graph process run for the same number of steps (which precisely builds a critical Erd\H{o}s--R\'enyi random graph). 

The proof of Proposition~\ref{prop:critical_bound}, which occupies the bulk of the paper, appears in Section~\ref{sec:critical_bd_proof}. It makes use of the connections between invasion percolation and critical percolation, and the fact that the components of the critical Erd\H{o}s--R\'enyi random graph are with high probability treelike, to reduce the analysis to that of a fragmentation process on large random binary trees. 

Finally, Section~\ref{sec:conc} proposes some future research directions suggested by the current work. 

\section{\bf A sketch proof of Theorem~\ref{thm:main}.}\label{sec:sketch}

\subsection{\bf Invasion percolation, Prim's algorithm, and Kruskal's algorithm}\label{sec:prim_kruskal}

Suppose that $G=(v(G),e(G),\rw)$ is a finite graph. If $\cS=\{v\}$ consists of a single vertex $v \in v(G)$, then invasion percolation is equivalent to {\em Prim's algorithm} \cite{prim57shortest} started from $v$, and $F(G,\cS)$ is thus the minimum-weight spanning tree (MST) of the weighted graph $G$. 
If $\cS$ consists of more than one vertex, the invasion percolation process can still be viewed as a form of Prim's algorithm, as follows. Augment $G$ by adding a new vertex $\rho$ and edges from $\rho$ to all elements of $\cS$. Fix $0 < \eps < \min(\rw_e,e \in e(G))$ and augment $\rw$ by  giving 
the edges $\{\rho x,x \in \cS\}$ each a distinct weight less than $\eps$. Write $G'_{\cS}=(v(G'_{\cS}),e(G'_{\cS}),\rw')$ for the augmented graph. Then invasion percolation on $G'_{\cS}$ with starting set $\{\rho\}$ will first add edges $\{\rho x,x \in \cS\}$, and will then add the same edges as invasion percolation on $G$ with starting set $\cS$, in the same order. It follows that the subgraph of $F(G'_{\cS},\rho)$ obtained by removing $\rho$ and its incident edges is precisely $F(G,\cS)$. 

In the setting of finite graphs, an alternative construction of $F(G,\cS)$ is given by {\em Kruskal's algorithm} \cite{kruskal56mst}, which works as follows.
Write $m=|e(G)|$ and list the edges of $G$ in increasing order of weight as $e(1),\ldots, e(m)$. 
Let $F_0=F_0^{G,\cS}=(v(G),\emptyset)$. Then, for $1 \le i \le m$: 
\begin{itemize}
\item If $e(i)=u(i)v(i)$ joins distinct connected components of $F_{i-1}$, and $u(i)$ and $v(i)$ do not both lie in components containing elements of $\cS$, then set $F_i=F_{i-1}+e(i):=(v(F_{i-1}),e(F_{i-1})\cup \{e(i)\})$. 
\item Otherwise, set $F_i=F_{i-1}$. 
\end{itemize}
The output of Kruskal's algorithm is the forest $F_m=F_m^{G,\cS}$. To see that $F_m=F(G,\cS)$, it suffices to consider running Kruskal's algorithm on the augmented graph $G'_{\cS}$ defined above. The result is the MST of $G'_{\cS}$, and is therefore equal to $F(G'_{\cS},\{\rho\})$. However, Kruskal's algorithm run on $G'_{\cS}$ and $\{\rho\}$ will begin by adding the edges $\rho x$ for $x \in \cS$, since these edges have lower weight than all other edges in $G'_{\cS}$. Once these edges are added, the vertices of $\cS$ all lie in a single connected component, so the remaining steps of Kruskal's algorithm run on $G'_{\cS}$ and $\{\rho\}$ add the same edges as Kruskal's algorithm run on $G$ and $\cS$, in the same order. It follows that $F_m$ can be obtained from $F(G'_{\cS},\{\rho\})$ by removing $\rho$ and its incident edges. We saw using Prim's algorithm that performing this operation to $F(G'_{\cS},\{\rho\})$ yields $F(G,\cS)$, and so indeed $F_m^{G,\cS}=F(G,\cS)$. 

It will be useful that the above construction couples the processes $(F_i^{G,\cS},0 \le i \le m)$ for different starting sets $\cS$: if $\cS'\subset \cS$ then $F_i^{G,\cS}$ is a subgraph of $F_i^{G,\cS'}$ for all $0 \le i \le m$. More specifically, suppose that $\cS=\cS' \cup \{z\}$ for some fixed $z\in v(G) \setminus \cS'$. Let $v \in \cS'$ be the unique element of $\cS'$ in the same component of $F_m^{G,\cS'}$ as $z$, and let $e(j)$ be the largest-weight edge on the path from $v$ to $z$ in $F_m^{G,\cS'}$. Then 
\begin{equation}\label{eq:coupling_edge}
F_i^{G,\cS} = \begin{cases}
				F_i^{G,\cS'}		&\mbox{ if }i < j\\
				F_i^{G,\cS'}-e(j)&\mbox{ if }i \ge j\, .
				\end{cases}
\end{equation}

\subsection{Kruskal's algorithm and the Erd\H{o}s--R\'enyi process}
There is a second useful coupling, between $(F_i^{G,\cS},0 \le i \le m)$ and a graph process which does not forbid cycles but maintains the condition that vertices in the starting set $\cS$ are not allowed to join the same connected component. The restricted process, which we call the {\em Erd\H{o}s--R\'enyi process} and denote $(G_i,0 \le i \le m)=(G_i^{\cS},0 \le i \le m)$, works as follows. List the edges of $G$ in increasing order of edge weight as $(e(i),0 \le i \le m)$. For $1 \le i \le m$, if the edge $e(i)$ joins connected components of $G_{i-1}$ containing distinct elements of $\cS$ then set $G_i=G_{i-1}=([n],e(G_{i-1}))$; otherwise, set $G_i=G_{i-1}+e(i)$. 
The final graph $G_{m}$ consists of $|\cS|$ connected components, each containing exactly one of the vertices of $\cS$. 

The orderings of edges in Kruskal's algorithm and in the Erd\H{o}s--R\'enyi process are identical. Moreover, if the same starting set $\cS$ is used for both processes, then the only edges which are added by the Erd\H{o}s--R\'enyi process but not by Kruskal's algorithm join vertices which already lie in the same connected component. 
It follows that $F_i^{G,\cS}$ and $G_i^{\cS}$ have the same connected components for all $1 \le i \le m$. (More strongly, for each connected component $C$ of $G_i^{\cS}$, the corresponding component of $F_i^{G,\cS}$ is the minimum weight spanning tree of $C$.) 
In particular, this yields that the size of the largest connected component is the same in $F(G,\cS)$ and in $G_{m}^{\cS}$. 

To justify the name ``Erd\H{o}s--R\'enyi process'', note that if $G=\rK_n$ is the randomly-weighted complete graph and $|\cS|=1$, then $(G_i^\cS,0 \le i \le m)=(G_i^\cS,0 \le i \le {n \choose 2})$ is precisely the classical Erd\H{o}s--R\'enyi random graph process, in which the edges of the complete graph are added one-at-a-time in exchangeable random order. 

\subsection{The critical random graph and the proof of Theorem~\ref{thm:main}}
We now specialize to the setting of this paper,  the randomly-weighted complete graph $\rK_n$. 
It is useful to continuize both the Erd\H{o}s--R\'enyi process and Kruskal's algorithm; write 
 $(G(n,\cS,p),0 \le p \le 1)$ for the random graph process in which  $G(n,\cS,p)$ has vertex set $[n]$ and edge set 
 \[
 \Big\{e \in e\big(\rK_{n,{n \choose 2}}^{\cS}\big): U_{e} \le p\Big\}\, ,
 \] 
 and $(F(n,\cS,p),0 \le p \le 1)$ for the process in which  $F(n,\cS,p)$ has vertex set $[n]$ and edge set 
 \[
 \Big\{e \in e\big(F_{n \choose 2}^{\rK_n,\cS}\big): U_{e} \le p\Big\}\, . 
 \]
 The continuous-time processes add the same edges as the discrete processes, and in the same order. 
More strongly, $G(n,\cS,U_{e(i)})=\rK_{n,i}^{\cS}$ for all $1 \le i \le {n \choose 2}$, and $(G(n,\cS,p),0 \le p \le 1)$ is constant except at times $(U_i,1 \le i \le {n \choose 2})$; the corresponding relation holds for the discrete- and continuous-time Kruskal processes.

When $|\cS|=1$ we omit $\cS$ from the notation, writing, e.g., $G(n,p)$ rather than $G(n,\cS,p)$, as in this case the processes do not in fact depend on $\cS$. Note that $F(n,1)$ is then the MST of $\rK_n$. 

The relation~\eqref{eq:coupling_edge} implies that for any $p \in (0,1)$ and $\cS\subset [n]$, the connected components of $F(n,\cS,p)$ refine those of $F(n,p)$, in that for any component $C$ of $F(n,p)$ the vertex set of $C$ may be written as a union of the vertex sets of components of $F(n,\cS,p)$. Since $F(n,\cS,p)$ and $G(n,\cS,p)$ have the same components for all $\cS \subset [n]$ and $p \in [0,1]$, the same fact holds for $G(n,\cS,p)$ and $G(n,p)$.

The heart of the proof that $M_n/n \to 0$ in probability when $k(n)/n^{1/3}\to\infty$ 
consists in establishing that in the {\em critical window} of the Erd\H{o}s-R\'enyi process, when $p=1/n+O(1/n^{4/3})$, the connected components of $F(n,[k(n)],p)$ all have size $o(n^{2/3})$ with high probability. For $\lambda \in \R$, write 
\[
p_{n,\lambda}=1/n+\lambda/n^{4/3}\, .
\] 
\begin{prop} \label{prop:critical_bound}
Fix positive integers $(k(n),n \ge 1)$ with $k(n) \in [n]$ and $k(n)/n^{1/3} \to \infty$. Next, fix $\lambda \in \R$, and let $M_{n,\lambda}(k(n))$ be the size of the largest connected component of $F(n,[k(n)],p_{n,\lambda})$. Then $M_{n,\lambda}(k(n))/n^{2/3} \to 0$ in probability. 
\end{prop}

The proof of Proposition~\ref{prop:critical_bound} appears in Section~\ref{sec:critical_bd_proof}. 
To prove Theorem~\ref{thm:main}, we combine this proposition with the following two pre-existing results about the structure of the minimum spanning tree of $\rK_n$. 
For $p \in [0,1]$, write $F^1(n,p)$ for the largest connected component of $F(n,p)$, with ties broken uniformly at random. Fix $\lambda \in \R$, and consider the forest obtained from the minimum spanning tree, $F(n,1)$, by removing the edges of $F^1(n,p_{n,\lambda})$. 
For each vertex $v$ of $F^1(n,p_{n,\lambda})$, write $T^n_{v,\lambda}$ for the tree of this forest containing $v$; see Figure~\ref{fig:forest_decomp}. Let $q^v_{n,\lambda}=|T^n_{v,\lambda}|/n$ be the proportion of vertices of $F(n,1)$ lying in $T^n_{v,\lambda}$. 
\begin{figure}[htb]
\[
\includegraphics[width=0.6\textwidth]{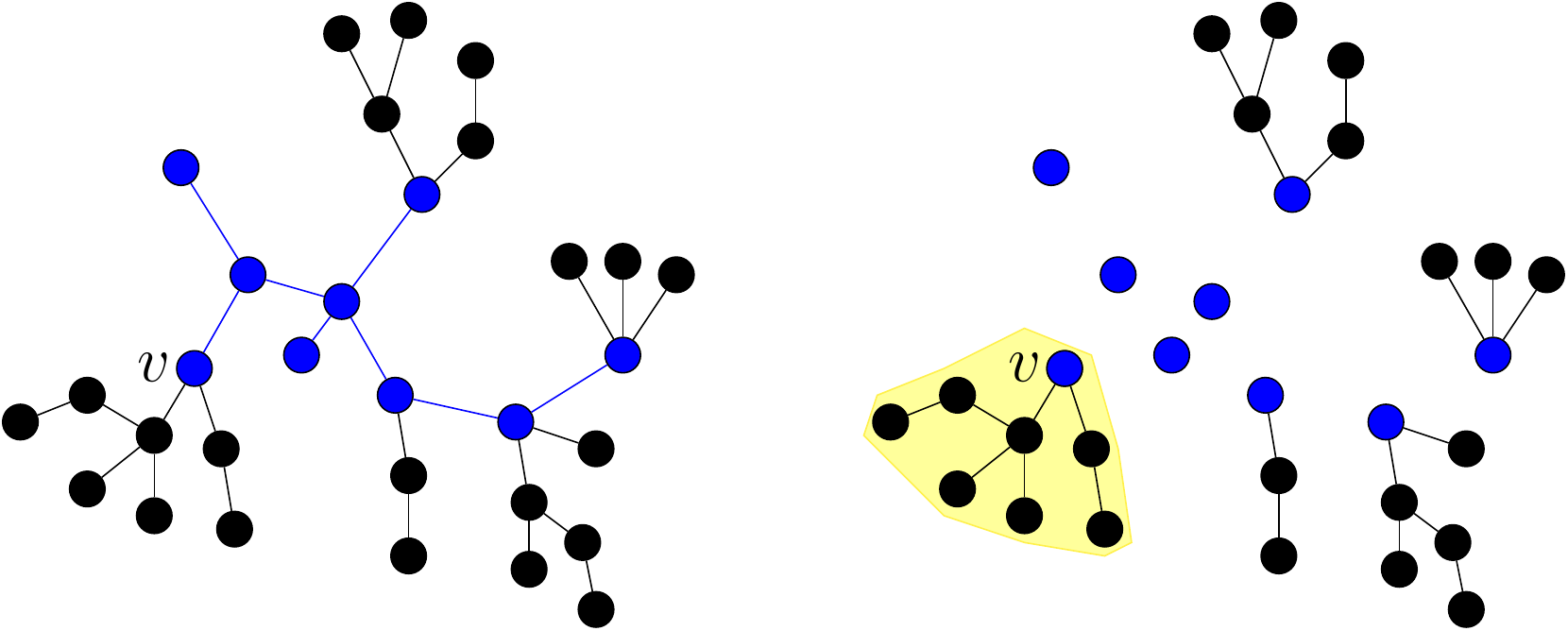}
\]
\caption{Left: An instantiation of $F(n,1)$ with $F^1(n,p_{n,\lambda})$ drawn in blue. Right: the forest obtained from $F(n,1)$ by removing the edges of $F^1(n,p_{n,\lambda})$. On the right the tree $T^n_{v,\lambda}$ is highlighted.}
\label{fig:forest_decomp}
\end{figure}
\begin{prop}[\cite{MR3706739}, Lemma 4.11]\label{prop:biggest_tree}
Write $\Delta^n_\lambda = \max(q^v_{n,\lambda}, v \in F^1(n,p_{n,\lambda}))$. Then for all $\delta > 0$, 
\[
\lim_{\lambda \to \infty}\limsup_{n \to \infty} \p{\Delta^n_\lambda > \delta} = 0\, .
\]
\end{prop}
Next, let
\[
\cF_{n,\lambda} := \sigma\big(U(e)\I{U(e) \le p_{n,\lambda}},e \in e(K_n)\big)=\sigma\big(U(e)\I{e \in e(G(n,p_{n,\lambda}))},e \in e(K_n)\big)
\] 
be the $\sigma$-algebra containing all information about the weights of edges in $G(n,p_{n,\lambda})$. 
\begin{prop}[\cite{MR4288328}, Lemma 6.19]\label{prop:exch_attachments}
For every $\lambda \in \R$, conditionally given $\cF_{n,\lambda}$, the collection of random variables $(q^v_{n,\lambda},v \in F^1(n,p_{n,\lambda}))$ is exchangeable. 
\end{prop}
The exchangeability in \cite[Lemma~6.19]{MR4288328} is stated conditionally given $G(n,p_{n,\lambda})$, rather than given $\cF_{n,\lambda}$. In other words, in \cite{MR4288328} the conditioning is only on the graph structure of $G(n,p_{n,\lambda})$, but not on the weights of its edges. However, an essentially identical proof to that given in~\cite{MR4288328} establishes the slightly stronger statement above. 

We also require a fact about concentration of exchangeable random sums, which is a consequence of a result of Aldous \cite{MR883646}.
\begin{prop}[\cite{MR883646}, Theorem 20.7]\label{prop:exc_lln}
For all $\eps > 0$ there exists $\delta > 0$ such that the following holds. Let $(q_i,1 \le i \le m)$ be non-negative real numbers with $\sum_{1 \le i \le m}q_i=1$, and let $(\pi(i),1 \le i \le m)$ be a uniformly random permutation of $[m]$. 
If $\max_{1 \le i \le m} q_i \le \delta$ then 
\[
\p{\max_{1 \le i \le m} \left|\sum_{j=1}^i \Big(q_{\pi(j)} - \frac{1}{m}\Big)\right|>\eps} < \eps\, .
\]
\end{prop}
See \cite[Lemma 7.5]{MR4084187} and \cite[Lemma 4.9]{MR3748328} for quantitative versions of this result. Proposition~\ref{prop:exc_lln}  is the last fact we need for the proof of our main result.
\begin{proof}[Proof of Theorem~\ref{thm:main}.]
As noted just after the statement of Theorem~\ref{thm:main}, the fact that $M_n/n \to 1$ in probability when $k(n)/n^{1/3} \to 0$ was proved in \cite{logan18variant}, so we need only handle the other assertion of the theorem. For the remainder of the proof we therefore assume that $k(n)/n^{1/3} \to \infty$. 

A result of {\L}uczak \cite[Theorem 3 ii.]{MR1099794} implies that if $p=p_n$ satisfies that $p_n=(1+o(1))/n$ and $n^{4/3}(p_n-1/n) \to \infty$, then the largest component of $G(n,p_n)$ has size $(2+o(1))np_n$ in probability. 
Since the components of $G(n,p_n)$ and of $F(n,p_n)$ are identical, 
recalling that $p_{n,\lambda}=1/n+\lambda/n^{4/3}$, it follows that if $\lambda=\lambda(n) \to \infty$ with $\lambda(n)=o(n^{1/3})$, then $|F^1(n,p_{n,\lambda(n)})|/(n^{2/3}\lambda(n)) \to 2$ in probability. By a subsubsequence argument, this implies that for all $\delta > 0$, 
\begin{equation}\label{eq:lc_size}
\lim_{\lambda \to \infty} \limsup_{n \to \infty} 
\p{|F^1(n,p_{n,\lambda})|/(n^{2/3}\lambda(n))<2-\delta} = 0.
\end{equation}

Fix $\eps \in (0,1)$. Then fix $\delta \in (0,\eps/2)$ small enough that Proposition~\ref{prop:exc_lln} holds for this $\eps$ and $\delta$, then let $\lambda$ be large enough that for all $n$ sufficiently large, 
\begin{equation}\label{eq:f1_bd}
\p{|F^1(n,p_{n,\lambda})|/n^{2/3} \le 1} < \eps\, 
\end{equation}
and
\begin{equation}\label{eq:deltanlambda_bd}
\p{\Delta^n_\lambda \ge \delta} < \eps\, .
\end{equation}
This is possible by~\eqref{eq:lc_size} and by  Proposition~\ref{prop:biggest_tree}. Since $\lambda$ is fixed, 
by Proposition~\ref{prop:critical_bound}, we  also have that 
\begin{equation}\label{eq:mnlambda_bd}
\p{M_{n,\lambda}(k(n))/n^{2/3} \ge \delta} < \eps\, 
\end{equation}
for $n$ sufficiently large. 

List the connected components of $F(n,[k(n)],p_{n,\lambda})$ contained in $F^1(n,p_{n,\lambda})$ as $C^1_{n,\lambda},\ldots,C^K_{n,\lambda}$; here $K$ is a random variable. These components are subtrees of $F^1(n,p_{n,\lambda})$, and their vertex sets partition $v(F^1(n,p_{n,\lambda}))$. 
Note that, writing $\mathcal{S}_{n,\lambda}=[k(n)] \cap v(F^1(n,p_{n,\lambda}))$, then each of $C^1_{n,\lambda},\ldots,C^K_{n,\lambda}$ contains exactly one vertex of $\mathcal{S}_{n,\lambda}$. 

We now consider the restricted process $(F(n,\cS_{n,\lambda},p),0 \le p \le 1)$. Due to the relation \eqref{eq:coupling_edge}, the only edges added in the Kruskal process $(F(n,p),0 \le p \le 1)$ which are not added in the restricted process 
$(F(n,\cS_{n,\lambda},p),0 \le p \le 1)$ are the edges of $F^1(n,p_{n,\lambda})$ which join distinct components $C^1_{n,\lambda},\ldots,C^K_{n,\lambda}$, and these edges are already present in $F(n,p_{n,\lambda})$. It follows that for each $1 \le i \le K$, the connected component of $F(n,\cS_{n,\lambda},1)$ containing $C^i_{n,\lambda}$ is precisely the union of the trees $\{T^n_{v,\lambda}, v \in v(C^i_{n,\lambda})\}$. On the other hand, since $\cS_{n,\lambda} \subset [k(n)]$, the components of $F(n,[k(n)],1)$ partition the components of $F(n,\cS_{n,\lambda},1)$, and so 
\begin{align*}
M_n 
& = \max(|C|:C\mbox{ is a component of }F(n,[k(n)],1))\\
& \le 
\max(|C|:C\mbox{ is a component of } F(n,S_{n,\lambda},1)) = 
\max_{1 \le i \le K} \sum_{v \in C^i_{n,\lambda}} |T^n_{v,\lambda}|\, .
\end{align*}

Write $m_n=|F^1(n,p_{n,\lambda})|$, then list the vertices of $F^1(n,p_{n,\lambda})$ 
as $v_1,\ldots,v_{m_n}$ so that for each $1 \le i \le K$, the vertices of $C^i_{n,\lambda}$ appear consecutively --- as
\[
v_{|C^1_{n,\lambda}|+\ldots+|C^{i-1}_{n,\lambda}|+1},\ldots,v_{|C^1_{n,\lambda}|+\ldots+|C^{i}_{n,\lambda}|}\, ,
\]
say. 
Necessarily $\max(|C^i_{n,\lambda}|,1 \le i \le K) \le M_{n,\lambda}(k(n))$, so 
on the event that $M_{n,\lambda}(k(n)) \le \delta n^{2/3}$ and $|F^1(n,p_{n,\lambda})| =m_n \ge n^{2/3}$, we then have 
\[
\frac{M_n}{n} \le \max\Big(\sum_{j=i}^{\ell} q^{v_j}_{n,\lambda}, 1 \le i < \ell \le m_n, \ell -i < \delta m_n\Big)\, .
\]
Therefore, on this event, if $M_n/n \ge 3\eps$ then we may find $i$ and $\ell$ as above so that 
\[
\sum_{j=i}^{\ell} \Big(q^{v_j}_{n,\lambda}-\frac{1}{m_n}\Big)
= \Big(\sum_{j=i}^{\ell} q^{v_j}_{n,\lambda}\Big)-\frac{\ell-i}{m_n}
\ge 3\eps-\delta >2\eps\, , 
\]
so either $\sum_{j=1}^i (q^{v_j}_{n,\lambda}-1/m_n) < -\eps$ or 
$\sum_{j=1}^\ell (q^{v_j}_{n,\lambda}-1/m_n) > \eps$. 
It follows that 
\begin{align}
\p{M_n \ge 3\eps n}
& \le \p{M_{n,\lambda}(k(n)) > \delta n^{2/3}} 
+ \p{|F^1(n,p_{n,\lambda})| < n^{2/3}} \notag\\
& + \p{\max_{1 \le i \le m_n}
 \Big|\sum_{j=1}^i \Big(q^{v_j}_{n,\lambda} - \frac{1}{m_n}\Big)\Big|>\eps} \notag\\
 & < 2\eps + 
 \p{\max_{1 \le i \le m_n}
 \Big|\sum_{j=1}^i \Big(q^{v_j}_{n,\lambda} - \frac{1}{m_n}\Big)\Big|>\eps}\, ,\label{eq:mn_bound}
\end{align}
where in the final line we have used~\eqref{eq:f1_bd} and~\eqref{eq:mnlambda_bd}. To bound the third probability we write 
\begin{align*}
&  \p{\max_{1 \le i \le m_n}
 \Big|\sum_{j=1}^i \Big(q^{v_j}_{n,\lambda} - \frac{1}{m_n}\Big)\Big|>\eps}\\
& =
 \E{ \p{\max_{1 \le i \le m_n}
 \Big|\sum_{j=1}^i \Big(q^{v_j}_{n,\lambda} - \frac{1}{m_n}\Big)\Big|>\eps~\Big\vert~\cF_{n,\lambda}}} \\
 & = 
 \E{ \p{\max_{1 \le i \le m_n}
 \Big|\sum_{j=1}^i \Big(q^{v_{\pi(j)}}_{n,\lambda} - \frac{1}{m_n}\Big)\Big|>\eps~\Big\vert~\cF_{n,\lambda}}}\, ,
\end{align*}
where conditionally given $\cF_{n,\lambda}$, $\pi$ is a uniformly random permutation of $m_n$ independent of the values $(q^{v_i}_{n,\lambda},1\le i \le m_n)$. The second equality holds as conditionally given $\cF_{n,\lambda}$ the random variables $(q^{v_i}_{n,\lambda},1\le i \le m_n)$ are exchangeable, due to Proposition~\ref{prop:exch_attachments}. 

Recall that $\Delta^n_\lambda := \max(q^v_{n,\lambda}, v \in F^1(n,p_{n,\lambda}))$; then by Proposition~\ref{prop:exc_lln} we have 
\begin{align*}
& \p{\max_{1 \le i \le m_n}
 \Big|\sum_{j=1}^i \Big(q^{v_{\pi(j)}}_{n,\lambda} - \frac{1}{m_n}\Big)\Big|>\eps~\vert~\cF_{n,\lambda}}\\
& \le 
 \p{\Delta^n_\lambda \ge \delta~\vert~\cF_{n,\lambda}} +
 \p{\max_{1 \le i \le m_n}
 \Big|\sum_{j=1}^i \Big(q^{v_{\pi(j)}}_{n,\lambda} - \frac{1}{m_n}\Big)\Big|>\eps~\vert~\cF_{n,\lambda},\Delta^n_\lambda \le \delta}\\
 & \le 
\p{\Delta^n_\lambda \ge \delta~\vert~\cF_{n,\lambda}}+\eps\, ,
\end{align*}
so it follows that 
\begin{align*}
\p{\max_{1 \le i \le m_n}
 \Big|\sum_{j=1}^i \Big(q^{v_{j}}_{n,\lambda} - \frac{1}{m_n}\Big)\Big|>\eps}
& \le \eps +\E{\p{\Delta^n_\lambda \ge \delta~\vert~\cF_{n,\lambda}}}\\
& = \eps + \p{\Delta^n_\lambda \ge \delta} < 2\eps\, ,
\end{align*}
the last inequality holding by~\eqref{eq:deltanlambda_bd}. Combining this bound with \eqref{eq:mn_bound}, it follows that $\p{M_n \ge 3\eps n} < 4\eps$; since $\eps > 0$ was arbitrary, this implies that $M_n/n \to 0$ in probability, as required.
\end{proof}

\section{\bf Proof of Proposition~\ref{prop:critical_bound}.}\label{sec:critical_bd_proof}

Our proof of Proposition~\ref{prop:critical_bound} has three steps. In the first step, we show that it suffices to prove that all components of $F(n,[k(n)],p_{n,\lambda})$ contained in the largest $O(1)$ components of $F(n,p_{n,\lambda})$ have size $o(n^{2/3})$ in probability. This essentially boils down to the application of well-known facts about the structure of the critical random graph.  In the second step we analyze the couplings presented above, between Kruskal's algorithm with different starting sets $\cS$ and between Kruskal's algorithm and the Erd\H{o}s-R\'enyi process. This analysis provides us with a tool for understanding how, distributionally, a given component $C$ of $F(n,p_{n,\lambda})$ is partitioned into pieces in $F(n,[k(n)],p_{n,\lambda})$, depending on the number of elements of $C \cap [k(n)]$. In the third step, which occupies most of the rest of the paper, we use the result of the analysis of the couplings to show that the largest connected components of $F(n,p_{n,\lambda})$ are indeed partitioned into pieces of size $o(n^{2/3})$ in $F(n,[k(n)],p_{n,\lambda})$, with high probability. 

\subsection{Step 1: reducing to the study of large components.}

For $p \in [0,1]$, list the components of $F(n,p)$ in decreasing order of size as $(F^i(n,p),i \ge 1)$, with ties broken uniformly at random. (The point of breaking ties this way is so that $v(F^i(n,p))$ is a uniformly random subset of $[n]$ conditional on its size.) Then for $\lambda \in \R$ and $\cS\subset [n]$, write $M^i_{n,\lambda}(\cS)$ for the size of the largest connected component of $F(n,\cS,p_{n,\lambda})$ contained in $F^i(n,p_{n,\lambda})$. If $\cS=[k(n)]$ we write $M^i_{n,\lambda}(k(n))$ instead of $M^i_{n,\lambda}([k(n)])$

In this section, we show how Proposition~\ref{prop:critical_bound} is a consequence of the following result.

\begin{prop}\label{prop:large_comps}
Fix $\lambda \in \R$ and $i \in \N$. If $k(n)/n^{1/3} \to \infty$ then $M^i_{n,\lambda}(k(n))/n^{2/3} \to 0$ in probability. 
\end{prop}
\begin{proof}[Proof of Proposition~\ref{prop:critical_bound}]
Fix $\eps > 0$. 

By \cite[Corollary 2]{aldous97brownian}, there is $j=j(\eps) \in \N$ such that for all $n \in \N$, 
\[
\p{\max(|F^\ell(n,p_{n,\lambda})|,\ell > j) > \eps n^{2/3}} < \eps\, .
\]
Since $M^\ell(n,p_{n,\lambda})(k(n)) \le |F^\ell(n,p_{n,\lambda})|$, 
it follows that for this value of $j$, 
\begin{align*}
\p{M_{n,\lambda}(k(n)) \ge \eps n^{2/3}}
& \le \p{\max_{1 \le \ell \le j} M^\ell_{n,\lambda}(k(n)) >\eps n^{2/3}} \\
		& + \p{\max(|F^\ell(n,p_{n,\lambda})|,\ell > j) > \eps n^{2/3}} \\
		& \le \eps + 
		\sum_{1 \le \ell \le j}
		\p{M^\ell_{n,\lambda}(k(n)) > \eps n^{2/3}} \\
		& \le 2\eps\, 
\end{align*}
for $n$ sufficiently large, the last bound holding due to Proposition~\ref{prop:large_comps}. Since $\eps > 0$ was arbitrary, this proves Proposition~\ref{prop:critical_bound}. 
\end{proof}

\subsection{Step 2: composing the couplings}
It is useful to briefly return to the setting of a deterministic connected graph $G=(v(G),e(G),\rw)$. Fix a starting set $\cS \subset v(G)$, and list edges of $G$ in increasing order of weight as $e(1),\ldots,e(m)$. Using the couplings of $F_i^{G,\cS}$ and $F_i^{G,\emptyset}$, on the one hand, and of $F_i$ and $G_i$, on the other hand, allows us to construct $F_m^{G,\cS}$ via a {\em path-and-cycle-breaking} process starting from $G$. 
Recall the definition of the augmented graph $G_\cS'$ from Section~\ref{sec:prim_kruskal}, which is formed from $G$ by adding a vertex $\rho$ which is joined to the vertices of $\cS$ by edges of very low weight. Then an edge $e(i)$ is added to $G_i$ but not to $F_i^{G,\emptyset}$ if and only if it lies on a cycle of $G_i$, which occurs if and only if it is the largest-weight edge on a cycle in $G$. the edge $e(i)$ is added to $F_i^{G,\emptyset}$ but not $F_i^{G,\cS}$ if and only if it is the largest-weight edge on a cycle in $G_{\cS}'$, which occurs if and only if there are distinct vertices $u,v \in \cS$ such that $e(i)$ lies on a path from $u$ to $v$ in $G_i$ (in which case $e(i)$ is the largest-weight edge on such a path). It follows that we may recover $F_m^{G,\cS}$ from $G$ as follows. 
\begin{itemize}
\item Let $H_0=G$.
\item For $0 \le i < m$, if either 
\begin{itemize}
\itemsep-0.1em
\item[(a)] $e(m-i)$ lies on a cycle of $H_i$, or 
\item[(b)] there exist distinct vertices $u,v \in \cS$ such that $e(m-i)$ lies on a path from $u$ to $v$ in $H_{i}$, 
\end{itemize}
then set $H_{i+1}=H_i-e(m-i)$; otherwise set $H_{i+1}=H_i$. 
\end{itemize}
The final graph $H_m$ is precisely $F_m^{G,\cS}$. 
This path-and-cycle-breaking construction of $F_m^{G,\cS}$ has the following immediate consequence in the setting of exchangeable edge weights. 
\begin{fact}[Path-and-cycle-breaking]
Let $G=(v(G),e(G),\rw)$ be a connected graph with exchangeable, almost surely distinct edge weights. Fix $\cS \subset v(G)$ and an ordering $\be=(e_1,\ldots,e_m)$ of $e(G)$. Generate a subgraph $F$ of $G$ as follows. 
\begin{enumerate}
\item 
Let $H_0=G$. 
\item For $0 \le i < m$, if $e_{i}$ lies on a cycle in $H_i$ or $e_i$ lies on a path in $H_i$ between distinct vertices of $\cS$, then set $H_{i+1}=H_i-e_i$; otherwise set $H_{i+1}=H_i$. 
\item Set $F=H_m$.
\end{enumerate}
If the ordering $\be$ is exchangeable then $F$ is distributed as $F_m^{G,\cS}$. 
\end{fact}
We call the above process {\em path-and-cycle-breaking on $G$ with starting set $\cS$ and edge ordering $\be$}, and refer to $F$ as the outcome of the process. (The edge weights $\rw$ are not used in the process, but they are used in defining $F_m^{G,\cS}$.) Most of our analysis will end up focussing on the case that $G$ is in fact a tree; in this case path-and-cycle-breaking process clearly never breaks cycles, and we simply refer to it as a path-breaking process.

We shall use the path-and-cycle-breaking process to understand how the components of $F(n,[k(n)],p_{n,\lambda})$ partition those of $G(n,p_{n,\lambda})$. Suppose that $C$ is a connected component of $G(n,p_{n,\lambda})$. 
Let $N=|v(C)|$ and let $S=|e(C)|-|v(C)|+1$ be the {\em surplus} of $C$. Let $C'$ be obtained from $C$ by relabeling the vertices of $C$ in increasing order as $1,\ldots,N$. Then $C'$ is uniformly distributed over connected graphs with vertex set $[N]$ and surplus $S$, and its edge weights are exchangeable. In view of these facts, the value of the next proposition should be rather clear.
\begin{prop}\label{prop:pacb}
For all $\eps > 0$ and any non-negative integer $s$, there exists integer $r > 0$ such that the following holds. For $q \ge 1$, let $G_q$ be uniformly distributed over the set of connected graphs with vertex set $[q]$ and surplus $s$. For $q \ge r$ let  $F_q=F_q(r,\be)$ be the outcome of the path-and-cycle-breaking process on $G_q$ with starting set $[r]$ and an exchangeable random ordering $\be=(e_1,\ldots,e_{m})$ of $e(G_q)$. Then for all $q$ sufficiently large, 
\[
\E{\max(|C|:C\mbox{ is a component of }F_q)} \le \eps q\, .
\]
\end{prop} 
This proposition has the following consequence. Fix non-negative integers $s$ and $(r(q),q \ge 1)$ with $r(q) \le q$ and with $r(q) \to \infty$ as $q \to \infty$. Let $G_q$ be as in Proposition~\ref{prop:pacb}, and let $F_q$ be the outcome of the path-and-cycle-breaking process on $G_q$ with starting set $[r(q)]$ and an exchangeable random ordering $\be=(e_1,\ldots,e_m)$ of $e(G_q)$. Then Proposition~\ref{prop:pacb} and Markov's inequality together imply that for any $\eps > 0$, 
\begin{equation}\label{eq:pacb_reformulation}
q^{-1}\E{\max(|C|:C\mbox{ is a component of }F_q)} \to 0\, 
\end{equation}
as $q \to \infty$. 

We prove Proposition~\ref{prop:pacb} in Section~\ref{sec:partitioning}, below; before doing so, we use it (or in fact its consequence, \eqref{eq:pacb_reformulation}) to prove Proposition~\ref{prop:large_comps}.
\begin{proof}[Proof of Proposition~\ref{prop:large_comps}]
Fix $\lambda \in \R$ and $i \in \N$. 
Write $G^i(n,p_{n,\lambda})$ for the component of $G^i(n,p_{n,\lambda})$ spanned by $F^i(n,p_{n,\lambda})$. 

Let $Q=Q(n)=|v(G^i(n,p_{n,\lambda}))|=|v(F^i(n,p_{n,\lambda}))|$, let $R=R(n)=|v(G^i(n,p_{n,\lambda}) \cap [k(n)]|$, and let $S=S(n)=|e(G^i(n,p_{n,\lambda}))|-|v(G^i(n,p_{n,\lambda})|+1$ be the surplus of $G^i(n,p_{n,\lambda})$. 

We will use in the course of the proof that $S(n)$ converges in distribution to an almost surely finite limit, and that 
$n^{-2/3}|F^i(n,p_{n,\lambda})|=n^{-2/3}|G^i(n,p_{n,\lambda})|$ 
converges in distribution to an almost surely finite, strictly positive limit; these facts appear in \cite[Folk Theorem 1 and Corollary 2]{aldous97brownian}. 

Conditionally given $Q(n)$, the vertex set $v(G^i(n,p_{n,\lambda}))$ is a uniformly random size-$Q(n)$ subset of $[n]$. The last convergence in distribution referenced in the previous paragraph (and in particular the fact that the limit is almost surely strictly positive) implies that for any $\eps > 0$ there exists $\delta > 0$ such that $\p{Q(n) \ge \delta n^{2/3}} > 1-\eps$. 
Since $k(n)/n^{1/3} \to \infty$, this implies that 
$k(n)Q(n)/n \to \infty$ in probability. 
Since $v(F^i(n,p_{n,\lambda}))$ is a uniformly random subset of $[n]$ conditional on its size, it then follows by standard concentration results for sampling without replacement that $R(n) \to \infty$ in probability. Moreover, for any fixed $s \in \N$, by \cite[Corollary 2]{aldous97brownian} we have $\liminf_{n \to \infty} \p{S(n)=s} >0$, Since $Q(n)$ and $R(n)$ both tend to infinity in probability, it follows that $Q(n)$ and $R(n)$ still tend to infinity in probability on the event that $S(n)=s$, in the sense that for any $x >0$, 
\[
\Cprob{Q(n)>x,R(n)>x}{S(n)=s} \to 1
\]
as $n \to \infty$.

Next, recall that 
\begin{align*}
& M^i_{n,\lambda}(k(n))\\
& :=\max(|C|:\mbox{$C$~a conn.\ comp.\ of $F(n,[k(n)],p_{n,\lambda})$ contained in $F^i(n,p_{n,\lambda})$})
\end{align*}
and write $M^i_{n,\lambda}=M^i_{n,\lambda}(k(n))$ for succinctness. 
Conditionally given $Q(n)$, $R(n)$ and $S(n)$, the random variable $M^i_{n,\lambda}$
has the same distribution as $\max(|C|:C\mbox{ is a component of }F_{Q(n)})$, where $F_{Q(n)}$ is the outcome of the path-and-cycle breaking process on $G_{Q(n)}$ with starting set $R(n)$. 
By the exchangeability of the vertex labels, this distribution is unchanged if rather than $R(n)$ we  use the starting set $[|R(n)|]=\{1,\ldots,|R(n)|\}$. It then follows from \eqref{eq:pacb_reformulation} and Markov's inequality that for any $\eps > 0$, 
\begin{align*}
\probC{M_{n,\lambda}^i > \eps |F^i(n,p_{n,\lambda})|}{S(n)=s} 
& =
\probC{Q(n)^{-1}M_{n,\lambda}^i > \eps}{S(n)=s}\\
& \to 0\, ,
\end{align*}
as $n \to \infty$. Moreover, since $S(n)$ converges in distribution to an almost surely finite limit, it follows that for all $\eps > 0$ there is $s_0$ such that for $n$ sufficiently large, $\p{S(n) > s_0} < \eps$. Combined with the preceding bound, this yields that for any $\eps > 0$, 
\begin{align*}
& \limsup_{n \to \infty} \p{M_{n,\lambda}^i > \eps |F^i(n,p_{n,\lambda})|}\\
& \le 
\limsup_{n \to \infty} \max_{1 \le s \le s_0} \probC{Q(n)^{-1}M_{n,\lambda}^i > \eps}{S(n)=s}
+ 
\limsup_{n \to \infty} \p{S(n) > s_0}\\
& \le \eps\, .
\end{align*}
It follows that $M_{n,\lambda}^i /|F^i(n,p_{n,\lambda})| \to 0$ in probability. Since $|F^i(n,p_{n,\lambda})|/n^{-2/3}$ converges in distribution to an almost surely finite limit, this implies that $M_{n,\lambda}^i/n^{2/3} \to 0$ in probability, as required. 
\end{proof}

\subsection{Step 3: partitioning a component}\label{sec:partitioning}
The goal of this section is to prove Proposition~\ref{prop:pacb}. We first prove the proposition for the special case $s=0$, in which case $G_q$ is a uniformly random tree with vertex set $q$, and the path-and-cycle-breaking process is simply a path-breaking process. We may restate the case $s=0$ of Proposition~\ref{prop:pacb} as follows. 
\begin{prop}\label{prop:path-breaking fixed r}
For all $\eps > 0$, there exists $r > 0$ such that the following holds. 
For $q \ge 1$, let $T_q$ be uniformly distributed over the set of trees with vertex set $[q]$. Let $F_q=F_q(r,\be)$ be the outcome of the path-breaking process on $T_q$ with starting set $[r]$ and an exchangeable random ordering $\be=(e_1,\ldots,e_{q-1})$ of $e(T_q)$. Then for all $q$ sufficiently large, 
\[
\E{\max(|C|:C\mbox{ is a component of }F_q)} \le \eps q\, .
\]
\end{prop}
In Section~\ref{sec:s=0} we prove Proposition~\ref{prop:path-breaking fixed r}, establishing the case $s=0$ of  Proposition~\ref{prop:pacb}.
We then use Proposition~\ref{prop:path-breaking fixed r} to handle the cases when $s \ge 1$, completing the proof of Proposition~\ref{prop:pacb}, in Section~\ref{sec:s>0}. 

For what follows it is useful to introduce the notation $u \lr{G} v$ to mean that there exists a path from $u$ to $v$ in graph $G$ (i.e., $u$ and $v$ are vertices of $G$ lying in the same connected component of $G$). 

\subsubsection{Proof of Proposition~\ref{prop:path-breaking fixed r}.}
\label{sec:s=0}
For $q \geq 1$, let $T_q$ be uniformly distributed over the set of trees with vertex set $[q]$. Let $F_q$ be the outcome of the path-breaking process on $T_q$ with starting set $[r]$ and an exchangeable random ordering $\textbf{e}=(e_1,\dots,e_{q-1})$ of $e(T_q)$. Then, for independent, uniformly random vertices $X,Y\in_u [q]$,
\begin{align} \label{eq:|C| bound}
\p{X \lr{F_q} Y}
&=
\E{\Cprob{X \lr{F_q} Y}{F_q}}\nonumber\\
&=
\E{\sum_{C \text{ is a component of } F_q} \frac{|v(C)|^2}{q^2}}\nonumber\\
&\geq
\E{\frac{\max(|v(C)| : C \text{ is a component of } F_q)^2}{q^2}} \nonumber\\
&\geq 
\frac{\E{\max(|v(C)| : C \text{ is a component of } F_q)}^2}{q^2} \,.
\end{align}
We thus analyze the probability that independent samples $X,Y\in_u[q]$ are connected in $F_q$. For the bulk of the analysis, it is in fact useful to instead consider $U,W$ sampled uniformly {\em without} replacement, from the set $\{r+1,\dots,q\}$. Since $\p{X = Y} + \p{X \in [r]} + \p{Y \in [r]} \to 0$ as $q \to \infty$, the error term this adds to the above bound is asymptotically negligible.

For a tree $t$ and a set $S \subseteq v(t)$, write $\hull{t}{S}$ for the smallest subtree of $t$ containing $S$. 

\begin{lem} \label{lem:path-breaking restriction}
Fix $q \ge r+2$, let $U,W$ be sampled uniformly without replacement from $\{r+1,\dots,q\}$, let $M_q=|e(\hull{T_q}{\{U,W\} \cup [r]})|$, and let $(e_{i_1},\dots,e_{i_{M_q}})$ be the restriction of the exchangeable random ordering $\textbf{e}$ to $e(\hull{T_q}{\{U,W\} \cup [r]})$. Then $U \lr{F_q} W$ if and only if $U \lr{F_q'} W$ where $F_q'$ is the outcome of the path-breaking process on $\hull{T_q}{\{U,W\} \cup [r]}$ with starting set $[r]$ and edge ordering $(e_{i_1},\dots,e_{i_{M_q}})$.
\end{lem}
\begin{proof}
For $i \not\in\{i_1,\ldots,i_{M_q} \}$, 
the edge $e_i$ does not lie on a path between any pair of elements of $[r]$, so is not removed by the path-breaking process on $T_q$ with starting set $[r]$ and edge ordering $(e_1,\ldots,e_{q-1})$. 
It follows (by induction) that 
the path-breaking process on $\hull{T_q}{\{U,W\}\cup [r]}$ with starting set $[r]$ and edge ordering $(e_{i_1},\dots,e_{i_{M_q}})$ removes the same edges, in the same order, as the previously mentioned path-breaking process on $T_q$. Hence, $F_q' = \hull{F_q}{\{U,W\} \cup [r]}$ 
so $U$ and $W$ are connected in $F_q$ if and only if they are connected in $F_q'$. 
\end{proof}
For ease of notation, let $T_q'=T_q'(r,\be)$ be the tree obtained from $\hull{T_q}{\{U,W\} \cup [r]}$ by relabeling $U,W$ as $r+1,r+2$, relabeling the vertices of $v\big(\hull{T_q}{\{U,W\} \cup [r]}\big) \setminus (\{U,W\} \cup [r])$ in increasing order as $\{r+3,\dots,M_q\}$, and relabeling the edges $(e_{i_1},\dots,e_{i_{M_q}})$ as $\be'=(e_1,\dots,e_{M_q})$. In a small abuse of notation we continue to denote the relabelings of $U,W$ by $U$ and $W$ rather than by $r+1$ and $r+2$. Finally, write $F_q'=F_q'(r,\be')$ for the outcome of the path breaking process on $T_q'$ with starting set $[r]$ and edge ordering $\be'$. 

For all $m \geq r+2$, let $\cT_m =\cT_m(r)$ be the set of trees with vertex set $[m]$ and with leaf set a subset of $[r+2]$. Note that since $T_q$ is a uniformly random tree with vertex set $[q]$, by symmetry, $T_q' \in_u \cT_{M_q}$, in the sense that for all $m \ge r+2$, conditionally given that $M_q=m$, then $T_q' \in_u \cT_m$. 

The definitions of this paragraph are illustrated in Figure~\ref{fig:uj}. 
For $T \in \cT_m$, let $P_0 = P_0(T) = \hull{T}{\{U,W\}}$, and for all $1 \leq i \leq r$ let $P_i = P_i(T)$ be the path in $T$ connecting $i$ to $\hull{T}{\{U,W\} \cup [i-1]}$. (If $i \in v(\hull{T}{\{U,W\} \cup [i-1]})$ then $P_i$ consists of a single vertex and no edges; otherwise $P_i$ has edge set $e(\hull{T}{\{U,W\} \cup [i]}) \setminus e(\hull{T}{\{U,W\} \cup [i-1]})$.) Then, for each $j \in \{1,2,3\}$ let $P_{0,j} = P_{0,j}(T)$ be the subpath of $P_0(T)$ with vertex set 
\[
\left\{ v \in v(P_0(T)) : \left\lfloor \dist_T(U,W) \frac{j-1}{3} \right\rfloor \leq \dist_T(U,v) \leq \left\lfloor \dist_T(U,W) \frac{j}{3} \right\rfloor \right\} \,,
\]
and let $\cU_j = \cU_j(T,r,q)$ be the subset of $[r]$ satisfying the following additional properties: for all $i \in \cU_j$,
\begin{enumerate}
\item $v(P_i) \cap v(P_{\ell}) = \emptyset$ for all $\ell \in [r] \setminus \{i\}$,
\item $P_i \cap P_0$ is a vertex in $P_{0,j}$, and
\item $1 \leq |e(P_i)| \leq \sqrt{q}.$
\end{enumerate}
Additionally, write $\cP_j = \left\{ P_i : i \in \cU_j \right\}$.
\begin{figure}[htb]
\includegraphics[width=0.9\textwidth]{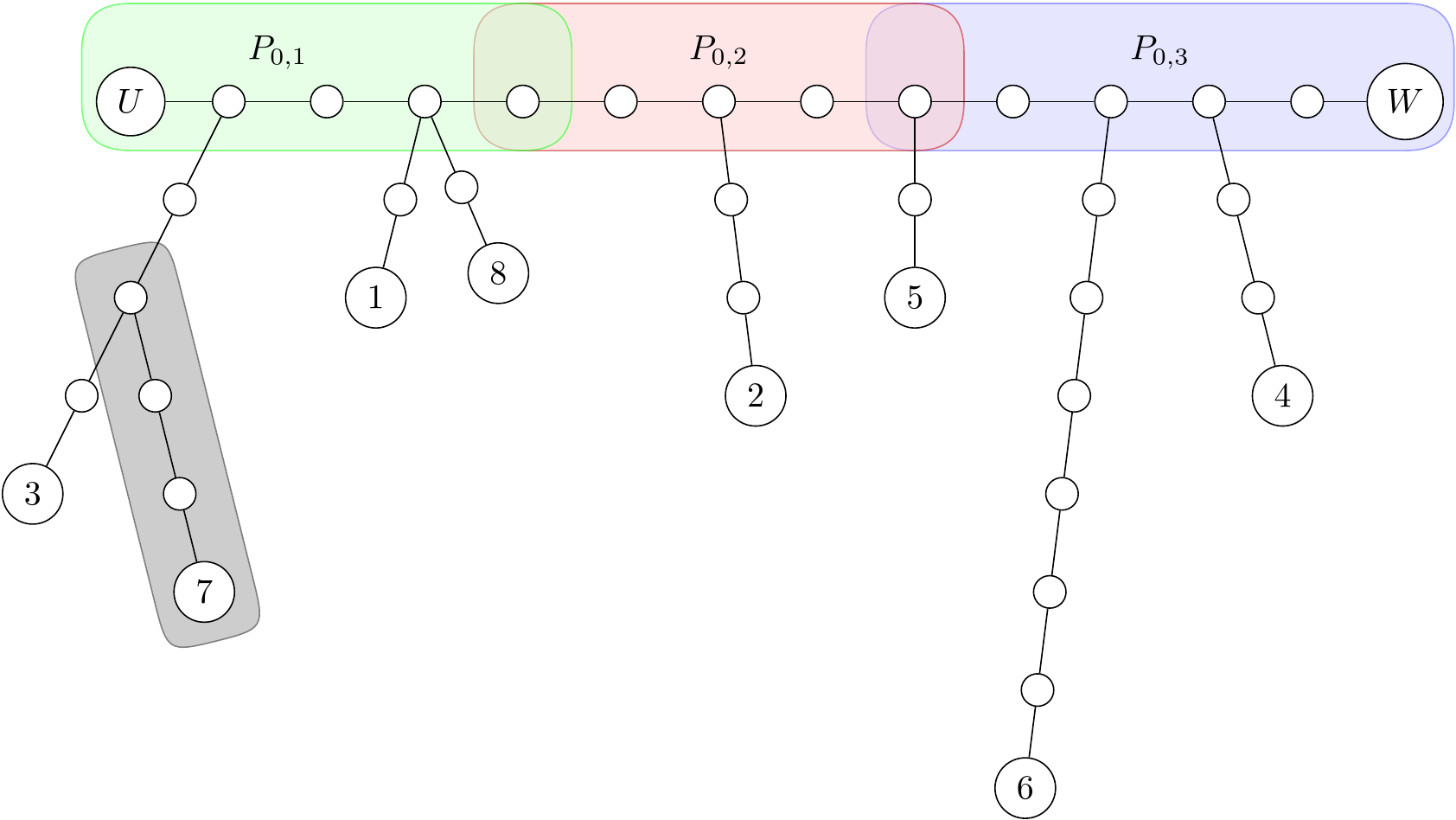}
\caption{An example of $T \in \cT_{40}$ with $r=8$ and $q=48$. In this example, $P_0$ connects vertices $U=r+1$ and $W=r+2$, the path $P_7$ is shaded in grey, and all the other paths $P_i$ connect the vertex labeled $i$ to the path $P_0$. 
Also, $\cU_1(T) = \emptyset$, $\cU_2(T) = \{2,5\}$, and $\cU_3(T) = \{4\}$. Note that $6 \notin \cU_3$ as $|e(P_6(T))| = 7 > \sqrt{48}$.} 
\label{fig:uj}
\end{figure}

For all paths $P \in T_q'$, say $P$ is {\em targeted} at time $t$ if $e_t \in e(P)$, say $P$ is {\em targeted for the first time} at time $t$ if $P$ is targeted at time $t$ and not before time $t$, and say $P$ is {\em broken} at time $t$ if $P$ is targeted at time $t$ and $e_t$ is removed during the path-breaking process. In the next lemma (and the subsequent Corollary \ref{cor:uv connected deterministic}), write $P_i = P_i(T_q')$ for $0 \leq i \leq r$, and write $P_{0,j} = P_{0,j}(T_q')$, $\cU_j = \cU_j(T_q')$ and $\cP_j = \cP_j(T_q')$ for $j \in \{1,2,3\}$.
\begin{lem}\label{lem:connection condition}
If $U \lr{F_q'} W$ and both $|\cU_1|\ge 1$ and $|\cU_3| \ge 1$, 
then all but at most one path in $\cP_1$ and all but at most one path in $\cP_3$ were targeted before $P_{0,2}$ was targeted for the first time.
\end{lem}
\begin{proof}
 Fix $1 \leq t \leq M_q$. Suppose that $P_i \in \cP_1$ and $P_j \in \cP_3$ have not been targeted by time $t$, and that $P_{0,2}$ is targeted for the first time at time $t$. Let $P_{i,j}$ be the path from $i$ to $j$ in $T_q'$. Then $P_{i,j} = P_i \cup Q_{0,1} \cup P_{0,2} \cup Q_{0,3} \cup P_j$ where $Q_{0,1} \subseteq P_{0,1}$ and $Q_{0,3} \subseteq P_{0,3}$. Hence, either $e_t \in P_{0,2}$ is cut, or an edge in $P_{i,j}$ has already been cut. In the latter case, since none of $P_i, P_j$, or $P_{0,2}$ have been targeted by time $t$, the edge that was already cut must be an edge of $Q_{0,1} \cup Q_{0,3} \subseteq P_0$. In both cases, an edge in $P_0$ is cut during the path-breaking process, meaning $U$ and $W$ are not connected in $F_q'$.
\end{proof}
For the next corollary, it is useful to introduce the shorthand
\[
\phat{\cdot} = \Cprob{\cdot}{T_q'} \,.
\]

\begin{cor}\label{cor:uv connected deterministic}
Let $E_q(p,u)$ be the event that $|e(P_{0,2})| \geq p$, $|\cU_1| \geq u$ and $|\cU_3| \geq u$. Then for $u \geq 1$,
\begin{align*}
\Cprobhat{U \lr{F_q'} W}{E_q(p,u)}
& \le \prod_{j=0}^{2u-3} \left( 1 - \frac{p}{p + \sqrt{q}(2u-j)} \right) \,.
\end{align*}
\end{cor}
\begin{proof}
Let $\cS = \cP_1 \cup \cP_3 \cup \{P_{0,2}\}$. For all paths $P \in \cS$, let $t_{P}$ be the first time $P$ is targeted. Then let $s = \big| \{P \in \cS : t_P < t_{P_{0,2}} \} \big|$. By Lemma \ref{lem:connection condition}, if $U$ and $W$ are connected in $F_q'$ and both $\cU_1$ and $\cU_3$ are nonempty, then $s \geq |\cS|-3$. 

Let $E_q^=(p,u)$ be the event that $|e(P_{0,2})| = p$, $|\cU_1| = u$ and $|\cU_3| = u$. List the paths in $\cS$ in the order they are first targeted as $P^{(1)},\dots,P^{(m)}$. By the exchangeability of the ordering $\be'=(e_1,\dots,e_{M_q})$, this list is a size-biased random ordering of the paths in $\cS$: that is, for all $1 \leq j \leq m$ and all $P \in \cS$,
\[
\Cprobhat{P^{(j)} = P}{P^{(1)},\dots,P^{(j-1)}} 
=
\frac{|P| \I{P \notin \{P^{(1)},\dots,P^{(j-1)}\}}}{\sum_{Q \in \cS} |Q| \I{Q \notin  \{P^{(1)},\dots,P^{(j-1)}\}}} \,.
\]
Since all paths in $\cS$ aside from $P_{0,2}$ have length at most $\sqrt{q}$, it follows that for all $0 \leq j \leq m$, 
\[
\Cprobhat{P^{(j)} = P_{0,2}}{E_q^=(p,u),P_{0,2} \notin \{P^{(1)},\dots,P^{(j-1)}\}}
\geq
\frac{p}{p + \sqrt{q}(2u-(j-1))} \,.
\]
Thus, by Bayes' formula, and since $m = 2u+1$ when $E^=_q(p,u)$ occurs, we have
\begin{align}
&\Cprobhat{s \geq |\cS|-3}{E_q^=(p,u)}
\nonumber\\
= \ &
\Cprobhat{P_{0,2} \notin \{P^{(1)},\dots,P^{(m-3)}\}}{E_q^=(p,u)}
\nonumber\\
= \ &
\prod_{j=0}^{m-4} \Cprobhat{P_{0,2} \ne P^{(j+1)}}{E_q^=(p,u),P_{0,2} \notin \{P^{(1)}, \dots ,P^{(j)}\}} 
\nonumber\\
\leq \ &
\prod_{j=0}^{2u-3} \left( 1 - \frac{p}{p + \sqrt{q}(2u-j)} \right) \,.\label{eq:eqbd}
\end{align}
Since the above product is decreasing in $p$ and in $u$, and 
$E_q(p,u)$ is the disjoint union of the events $\{E_q^=(p',u'),p' \ge p, u' \ge u\}$, the bound claimed in the corollary follows from \eqref{eq:eqbd} by the law of total probability.
\end{proof}
In order to make use of Corollary \ref{cor:uv connected deterministic}, we now must analyze the typical behaviour of $|e(P_{0,2(T_q')})|$, $|\cU_1(T_q')|$ and $|\cU_3(T_q')|$. We first gather some auxiliary facts which are crucial for this analysis. 

The first facts relate to the asymptotic structure of $T_q'(r)$ for $q$ large. This is described by the so-called {\em line-breaking construction} of Aldous \cite{aldous91continuum1}. 
Let $(\pi_i,i \ge 0)$ be the ordered sequence of inter-arrival times of a Poisson point sequence on $[0,\infty)$ with intensity measure $\lambda(t)=t$. Such a process may be concretely realized as follows. Let $(E_i, i \geq 1)$ be independent $\mathrm{Exp}(1)$ random variables, let $\pi_0 = \sqrt{2} E_1^{1/2}$ and, for each $j \ge 1$, let $\pi_j = \sqrt{2} \left( E_1 + \dots + E_{j+1} \right)^{1/2} -  \sqrt{2} \left( E_1 + \dots + E_{j} \right)^{1/2}$. The atoms of the Poisson process are thus located at the points $\left(\sqrt{2}(E_1+\ldots+E_i)^{1/2},i \ge 1\right)$. 

Now for $r \ge 0$, construct a binary tree $T_\infty(r)$ 
with edge lengths and with leaf labels $U,W$ and $1,\ldots,r$, as follows. The tree $T_\infty(0)$ is a line segment $P_{\infty,0}$ of length $|P_{\infty,0}|=\pi_0$, with endpoints labeled $U$ and $W$. Inductively, for $r \ge 1$ the tree $T_\infty(r)$ is constructed from $T_{\infty}(r-1)$ by attaching a line segment $P_{\infty,r}$ of length $|P_{\infty,r}|
=\pi_r$ to a uniformly chosen point of $T_{\infty}(r-1)$, and assigning label $r$ to the new leaf at the far end of the line segment.

The next proposition is a consequence of \cite[Theorem 8]{aldous91continuum1}. In what follows, if $G$ is an unweighted graph then we write $\dist_G(x,y)$ to mean the graph distance between vertices $x$ and $y$ in $G$ (the fewest number of edges in an $x-y$ path). If $G$ is a graph with edge lengths then we write $\dist_G(x,y)$ to mean the length of the shortest $x-y$ path, taking edge lengths into account.
\begin{prop} \label{prop:branch lengths}
As $q \to \infty$, 
\begin{align*}
& (q^{-1/2}\mathrm{dist}_{T_q'(r)}(x,y):x,y \in \{U,W\} \cup \{1,\ldots,r\})\\
& \convdist 
(\mathrm{dist}_{T_\infty(r)}(x,y):x,y \in \{U,W\} \cup \{1,\ldots,r\})\, .
\end{align*}
Moreover, for all $r \ge 1$, ignoring its edge lengths, the tree $T_{\infty}(r)$ is uniformly distributed over the set of binary trees with leaf labels $\{U,W\}\cup [r]$. Finally, for any permutation 
$\phi:\{U,W,1,\ldots,r\} \to \{U,W,1,\ldots,r\}$, the tree $T^\phi_{\infty}(r)$ obtained from $T_{\infty}(r)$ by relabeling its leaves  according to the permutation $\phi$ has the same law as $T_\infty(r)$. 
\end{prop}
Write $\alpha_r$ for the point of $T_{\infty}(r-1)$ to which $P_{\infty,r}$ is attached. Note that the lengths of the paths $(P_{\infty,i},0 \le i \le r)$ may be recovered from $T_{\infty,r}$ as 
$\pi_i = \dist_{T_\infty(r)}(i,\alpha_i)$. 
Thus, the convergence in Proposition~\ref{prop:branch lengths} directly implies that 
\[
(q^{-1/2}|e(P_0)|,\dots,q^{-1/2}|e(P_r)|) \stackrel{\mathrm{d}}{\to} 
(|P_{\infty,0}|,\ldots,|P_{\infty,r}|)=
(\pi_0,\dots,\pi_r)\]
as $q \to \infty$. 

For later use it's handy to describe the reverse construction (recovering the branches from the tree) in a little more generality. Given a binary tree $t$ with edge lengths and with leaves labeled by the elements of $\{U,W\}\cup [r]$, we define a growing sequence of subtrees $t(0),\ldots,t(r)$ where $t(i)$ is the smallest subtree of $t$ containing the leaves in $\{U,W\}\cup \{1,\ldots,i\}$. We then let $P_0(t)=t(0)$, for $i \in [r]$ we let $P_i(t)$ be the path connecting the leaf $i$ to the subtree $t(i-1)$, and let $\alpha_i(t)$ be the point of $t(i-1)$ to which $P_i(t)$ attaches. With these definitions, we have $P_{\infty,i}=P_i(T_{\infty,r})$ for $0 \le i \le r$.

The following tail bound for $|e(P_{0,2}(T_q'))|$ is a straightforward consequence of Proposition~\ref{prop:branch lengths}. 
\begin{cor} \label{cor:P02}
For all $\eps > 0$ and positive integers $r$, for large enough $q$,
\[
\p{|e(P_{0,2}(T_q'))| < \eps \sqrt{q}}
\leq
5\eps^2 \,.
\]
\end{cor}
\begin{proof}
By Proposition \ref{prop:branch lengths}, we have that 
\begin{align*}
\p{|e(P_0(T_q'))| < 3 \eps \sqrt{q}}
&=
(1+o_q(1)) \p{\sqrt{2}(E_1)^{1/2} < 3 \eps}\\
&=
(1+o_q(1)) \p{E_1 < \frac{9 \eps^2}{2}}\\
&=
(1+o_q(1)) \left( 1 - e^{-\frac{9 \eps^2}{2}} \right)\\
&\le
(1+o_q(1)) \frac{9 \eps^2}{2} \,. 
\end{align*}
The result now follows from the fact that $|e(P_{0,2}(T_q'))| \geq \frac{1}{3} |e(P_0(T_q'))| - 1$.
\end{proof}
The next lemma states the tail bound we need for $|\cU_1(T_q')|$ and $|\cU_3(T_q')|$.
\begin{lem} \label{lem:|U1|}
For all $\eps > 0$ sufficiently small, there exists a positive integer $r_0$ such that for all $r \ge r_0$, for all $q$ sufficiently large, with $T_q'=T_q'(r,\be)$, 
\[
\p{|\cU_1(T_q')| < \eps \sqrt{r}} < 23\eps \,,
\]
and the same bound holds with $|\cU_1(T_q')|$ replaced by $|\cU_3(T_q')|$. 
\end{lem}
The proof of Lemma~\ref{lem:|U1|} is somewhat involved, so before proving it we first use it together with the preceding results of the section to prove Proposition~\ref{prop:path-breaking fixed r}.

\begin{proof}[Proof of Proposition~\ref{prop:path-breaking fixed r}.]
In this proof, for readability we omit insignificant floors and ceilings.
Fix $\eps > 0$ small enough that Lemma~\ref{lem:|U1|} applies, let $r_0$ be 
as in Lemma~\ref{lem:|U1|}, and fix $r \ge r_0$ large enough that 
$\sum_{j=4}^{\eps r^{1/2}} \tfrac{1}{j} \ge \eps^{-1}\log(\eps^{-1})$. For the duration of the proof, write $P_j = P_j(T_q')$ for all $0 \leq j \leq r$ and $P_{0,i} = P_{0,i}(T_q')$ and $\cU_i = \cU_i(T_q')$ for $i \in \{1,2,3\}$. 

Recall from Corollary~\ref{cor:uv connected deterministic} that 
$E_q(p,u)$ is the event that $|e(P_{0,q})| \ge p$, $|\cU_1| \ge u$ and $|\cU_3| \ge u$. 
Taking $p=\eps q^{1/2}$ and $u=\eps r^{1/2}/2$, by that corollary and Bayes' formula we have 
\begin{align*}
\phat{U \lr{F_q'} W}
& \le 
\Cprobhat{U \lr{F_q'} W}{E_q(p,u)}
+
\phat{E_q(p,u)^c}\\
& 
\le
\prod_{j=0}^{2u-3} \left( 1 - \frac{p }{p + q^{1/2}(2u-j)} \right) +\phat{E_q(p,u)^c}\\
& = 
\prod_{j=0}^{\eps r^{1/2}-3} \left( 1 - \frac{\eps}{\eps(1+r^{1/2})-j} \right)+\phat{E_q(p,u)^c}\\
& \le 
\exp\pran{-\eps \sum_{j=0}^{\eps r^{1/2}-3} \frac{1}{\eps(1+r^{1/2})-j}}+\phat{E_q(p,u)^c}
\\
& \le 
\exp\pran{-\eps \sum_{j=4}^{\eps r^{1/2}} \frac{1}{j}}+\phat{E_q(p,u)^c}\\
& \le 
\eps + \phat{E_q(p,u)^c}\, ,
\end{align*}
the last bound holding since we chose $r$ large enough that $\sum_{j=4}^{\eps r^{1/2}} \tfrac{1}{j} \ge \eps^{-1}\log(\eps^{-1})$. 
Taking expectations, the tower law yields that 
\[
\p{U \lr{F_q'} W}
\le \eps + \p{E_q(p,u)^c}. 
\]
By Corollary~\ref{cor:P02} we have 
$\p{|e(P_{0,2})| \le \eps q^{1/2}} \le 5\eps^2$ and by Lemma~\ref{lem:|U1|} we have $\p{\min(|\cU_1|,|\cU_3|) \le \eps r^{1/2}/2} \le 46 \eps$, so 
$\p{E_q(p,u)^c} \le 5\eps^2+46 \eps$. 
Also, by Lemma \ref{lem:path-breaking restriction} we have $\p{U \lr{F_q'} W} = \p{U \lr{F_q} W}$, so we obtain the bound
\[
\p{U \lr{F_q} W} \le 47\eps + 5\eps^2 \, .
\]

To conclude, let $X,Y$ be independent uniform samples from $[q]$. The conditional distribution of $(X,Y)$ given that $X \ne Y$ and that $X \not\in [r],Y \not\in[r]$ is precisely that of $(U,V)$, so 
\begin{align*}
\p{X \lr{F_q} Y}& \le 
\p{U \lr{F_q} W}\!+\p{X=Y}\!+\p{X \in [r]}\!+\p{Y \in [r]} \\
& \le 47 \eps + 5 \eps^2 + \frac{2r+1}{q} \,.
\end{align*}
Finally, by \eqref{eq:|C| bound}, we have 
\[
\p{X \lr{F_q} Y} \geq \frac{\E{\max\left( |C| : C \text{ is a component of } F_q \right)}^2}{q^2},
\]
and so 
\begin{align*}
\E{\max\left( |C| : C \text{ is a component of } F_q \right)} 
&\leq
q \p{X \lr{F_q'} X}^{1/2}\\
&\le q\left(47 \eps + 5\eps^2+ \frac{2r+1}{q}\right)^{1/2} \,.
\end{align*}
The result follows since $\eps>0$ can be taken arbitrarily small, and since we can make $(2r+1)/q$ as small as we like by taking $q$ large. 
\end{proof}

The remainder of the section is devoted to proving Lemma~\ref{lem:|U1|}. We use Proposition~\ref{prop:branch lengths} to allow us to control the large-$q$ behaviour of the probabilities in question by instead studying the limiting tree $T_{\infty}(r)$. 
Given a binary tree $t$ with edge lengths and with leaves labeled by $\{U,W\} \cup[r]$, 
recall that $\alpha_t(i)$ is the attachment point of the line segment $P_{i}(t)$ to $t(i-1)$, and that $|P_{i}(t)|=\dist_{t}(i,\alpha(i))$. 
Let $\cU_1(t)$ be the set of leaves $i \in [r]$ satisfying the following properties. 
\begin{enumerate}
\item $P_{i}(t) \cap P_{j}(t) = \emptyset$ for all $j \in [r]\setminus \{i\}$,
\item $\alpha_i(t) \in P_0(t)$ and $\dist_{t}(\alpha_i(t),U) \le \dist_{t}(U,W)/3$, and
\item $\dist_{t}(i,\alpha_i(t)) \le 1$\, .
\end{enumerate}
Define $\cU_3(t)$ in the same way, but with the second condition replaced by the condition that $\dist_{t}(\alpha_i(t),W) \le \dist_{t}(U,W)/3$.

The convergence in Proposition~\ref{prop:branch lengths} implies that for any $r \ge 1$, as $q \to \infty$ we have $(|\cU_1(T_q',r,q)|,|\cU_3(T_q',r,q)| \convdist (|\cU_1(T_{\infty}(r))|,|\cU_3(T_{\infty}(r))|)$. This allows us to prove the proposition by proving lower tail bounds for $|\cU_1(T_{\infty}(r)|$ and $|\cU_3(T_{\infty}(r)|$. To establish such bounds, we will use the second moment method, applied conditionally given $|\pi_0|$. 
The application of the method is greatly simplified by the following exchangeability result, which allows us to focus our attention on the final two paths $P_{\infty,r-1}$ and $P_{\infty,r}$
\begin{lem}\label{lem:U1_exchange}
Write $\cU_1=\cU_1(T_\infty(r))$. Then for all $i \in [r]$, 
\begin{align*}
& \Cprob{P_{\infty,i}\in\cU_1}{\pi_0} = \Cprob{P_{\infty,r}\in\cU_1}{\pi_0}\,,
\end{align*}
and for all $i,j \in [r]$ with $i \neq j$, 
\begin{align*}
\Cprob{P_{\infty,i}\in\cU_1,P_{\infty,j}\in\cU_1}{\pi_0}
&=
\Cprob{P_{\infty,r-1}\in\cU_1,P_{\infty,r}\in\cU_1}{\pi_0}\,.
\end{align*}
Moreover, the same identities hold with 
$\cU_1$ replaced by $\cU_3=\cU_3(T_\infty(r))$. 
\end{lem}
\begin{proof}
We work on the probability-one event that $|P_{\infty,i}|> 0$ for all $i \in [r]$. (Note that $P_{\infty,i}=P_i(T_\infty(r))$. 
Fix any permutation $\phi$ of the leaf labels $\{1,\ldots,r\}$, and let $T_\infty^\phi(r)$ be the tree obtained from $T_\infty(r)$ by permuting the labels $\{1,\ldots,r\}$ according to $\phi$; note that labels $U$ and $W$ remain fixed. Any such permutation  induces an automorphism of $T_\infty(r)$ and $T_\infty^\phi(r)$ as binary leaf-labeled trees with edge lengths.

We claim that $\cU_1(T^\phi_\infty(r))=\{\phi(i): i \in \cU_1(T_\infty(r))\}$. 
To see this, fix $i \in [r]$. If $i \in \cU_1(T_\infty(r))$, then $P_{\infty,i} \cap P_{\infty,j}=\emptyset$ for all $j \in [r]\setminus i$, so the only point of intersection of $P_{\infty,i}$ with the rest of $T_\infty(r)$ lies on the path $P_{\infty,0}=P_0(T_\infty(r))$ from $U$ to $W$. 
Writing $P_{\infty,i}^\phi$ for the image of $P_{\infty,i}$ in $T_\infty^\phi(r)$ under the automorphism induced by $\phi$, the only point of intersection of $P_{\infty,i}^\phi$ with the rest of $T_\infty^\phi(r)$ must then lie on the path from $U$ to $W$ in $T_\infty^\phi(r)$, since the labels $U$ and $W$ are unchanged by $\phi$. Thus, $P_{\infty,i}^\phi= P_{\phi(i)}(T_\infty^\phi(r))$, and so $P_{\phi(i)}(T_\infty^\phi(r)) \cap P_j(T_\infty^\phi(r))=\emptyset$ for all $j \in [r]\setminus \{\phi(i)\}$. Since the lengths of $P_{\infty,i}$ and $P_{\infty,i}^\phi$, and their attachment points to the $U-W$ path, are the same in $T_\infty(r)$ and $T_\infty^\phi(r)$, it follows that $\phi(i) \in \cU_1(T_\infty^\phi(r))$. A corresponding argument using $\phi^{-1}$ shows that if $i \in \cU_1(T_\infty^\phi(r))$ then $\phi^{-1}(i) \in \cU_1(T_\infty(r)$, which establishes the claim. 

By Proposition~\ref{prop:branch lengths}, for any permutation $\phi:[r] \to [r]$, the trees $T_\infty^\phi(r)$ and $T_\infty(r)$ have the same law. Since $\cU_1(T^\phi_\infty(r))=\{\phi(i): i \in \cU_1(T_\infty(r))\}$, by taking $\phi$ to be a uniformly random permutation of $[r]$ it then follows that for all $i \in [r]$, and $0 \le s \le r$, 
\[
\Cprob{i \in \cU_1(T_\infty(r))}{|\cU_1(T_\infty(r))|=s} = \frac{s}{r}
\]
and hence 
\[
\Cprob{i \in \cU_1(T_\infty(r))}{|\cU_1(T_\infty(r))|=s}=\Cprob{r \in \cU_1(T_\infty(r))}{|\cU_1(T_\infty(r))|=s} \,.
\]
Similarly, for all $1 \le i<j\le r$, 
\[
\Cprob{i,j \in \cU_1(T_\infty(r))}{|\cU_1(T_\infty(r))|=s} = \frac{s(s-1)}{r(r-1)},
\]
and hence 
\begin{align*}
& \Cprob{i,j \in \cU_1(T_\infty(r))}{|\cU_1(T_\infty(r))|=s}\\
& =
\Cprob{r-1,r \in \cU_1(T_\infty(r))}{|\cU_1(T_\infty(r))|=s}\, .
\end{align*}
The lemma now follows by averaging over $s=|\cU_1(T_\infty(r))|$. 
\end{proof}

\begin{proof}[Proof of Lemma~\ref{lem:|U1|}]
As mentioned earlier, the convergence in distribution from Proposition~\ref{prop:branch lengths} implies that for any $r \ge 1$, as $q \to \infty$ we have $(\cU_1(T_q',r,q),\cU_3(T_q',r,q) \convdist (\cU_1(T_{\infty}(r)),\cU_3(T_{\infty}(r)))$. To prove the lemma it thus suffices to show that 
for all $\eps > 0$ there exists $r_0$ such that for all $r \ge r_0$, 
\begin{equation}\label{eq:toprove}
\p{|\cU_1(T_{\infty}(r))| < \eps \sqrt{r}} < 22 \eps.
\end{equation}
(The same bound then holds for $\cU_3(T_{\infty}(r))$ by symmetry.) The remainder of the proof is thus devoted to establishing \eqref{eq:toprove}. In what follows we write $\cU_1=\cU_1(T_\infty(r))$. 

By Lemma~\ref{lem:U1_exchange}, we have 
\[
\Cexp{\cU_1}{\pi_0}
= 
r\probC{r \in \cU_1}{\pi_0} \,.
\]
On the probability-one event that $\alpha(1),\ldots,\alpha(r)$ are all distinct, $r \in \cU_1$ if and only if $\pi_r \le 1$, $\alpha(r) \in P_{\infty,0}$, and $\mathrm{dist}_{T_{\infty}(r)}(\alpha(r),U) \leq \pi_0 /3$. Since $\alpha(r)$ is uniformly distributed over $T_{\infty}(r-1)$, which is the union of the paths $P_{\infty,0},\ldots,P_{\infty,r-1}$, for $r > 1$ we thus have 
\begin{align*}
\probC{r \in \cU_1}{E_1}& = 
\probC{r \in \cU_1}{\pi_0}\\
& = \Cexp{\probC{r \in \cU_1}{\pi_0,\ldots,\pi_{r}}}{\pi_0}\\
& = \Cexp{\frac{\pi_0/3}{|\pi_0|+\ldots+|\pi_{r-1}|}\I{|\pi_r| \le 1}}{\pi_0}\\
& = \Cexp{\frac{1}{3}\frac{E_1^{1/2}}{(E_1+\ldots+E_{r})^{1/2}}\I{\pi_r \le 1}}{E_1} \,.
\end{align*}
For the first and last identities above, we used that conditioning on $\pi_0$ and on $E_1$ is equivalent, since $\pi_0=\sqrt{2}E_1^{1/2}$.

Likewise, still on the event that $\alpha(1),\ldots,\alpha(r)$ are all distinct, provided that $r \ge 2$, the point $r-1$ belongs to $\cU_1$ if and only if $\pi_{r-1} \le 1$, $\alpha(r-1) \in P_{\infty,0}$, and $\mathrm{dist}_{T_{\infty}(r)}(\alpha(r-1),U) \leq \pi_0 /3$. A similar derivation then shows that for $r \ge 2$, 
\begin{align*}
& 
\probC{r-1\in \cU_1,r \in \cU_1}{E_1}\\
& = 
\Cexp{\frac{1}{9} \frac{E_1^{1/2}}{(E_1+\ldots+E_{r-1})^{1/2}}\frac{E_1^{1/2}}{(E_1+\ldots+E_{r})^{1/2}}\I{\pi_{r-1},\pi_r \le 1}}{E_1} \,.
\end{align*}

We let 
\[
R=\frac{1}{3}\frac{E_1^{1/2}}{(E_1+\dots +E_r)^{1/2}}
\]
and 
\[
R'=\frac{1}{9}\frac{E_1}{(E_1+\dots +E_{r-1})^{1/2} (E_1+\dots +E_r)^{1/2}}\, ,
\]
so that the above identities may be written more succinctly as 
\begin{equation}\label{eq:first_mt}
\probC{r \in \cU_1}{E_1}=\Cexp{R\I{\pi_r \le 1}}{E_1}
\end{equation}
and
\begin{equation}\label{eq:second_mt}
\probC{r-1\in \cU_1,r \in \cU_1}{E_1}=\Cexp{R'\I{\pi_{r-1},\pi_r \le 1}}{E_1}
\end{equation}

Now fix $\eps \in (0,1/2)$ small and let $A(\eps,r)$ be the event that 
\[
(1-\eps)r \le E_2+\ldots+E_{r-1}\le 
E_2+\ldots+E_r\le (1+\eps) r
\]
and that $E_{r} \le r^{1/2}$ and $E_{r+1} \le r^{1/2}$. On $A(\eps,r)$, using the bound $(a+x)^{1/2} - a^{1/2} \leq x/(2a^{1/2})$ for all $a,x > 0$, we have that
\begin{align*}
\pi_{r}
& = \sqrt{2}\left((E_1+\ldots+E_{r+1})^{1/2}-(E_1+\ldots+E_{r})^{1/2}\right)\\
& \le \sqrt{2}\left((E_2+\ldots+E_{r+1})^{1/2}-(E_2+\ldots+E_{r})^{1/2}\right)\\ 
& \le 
\sqrt{2} \frac{E_{r+1}}{2((1-\eps)r)^{1/2}} < 1\, ,
\end{align*}
and likewise $\pi_{r-1} < 1$ on $A(\eps,r)$.

For $r$ large enough, $\p{E_{r} \ge r^{1/2}} =e^{-r^{1/2}} \le 1/r^2$ and $\p{E_{r+1} \ge r^{1/2}} \le 1/r^2$, 
and by Chebyshev's inequality, 
\[
\p{E_2+\ldots+E_{r-1} \le (1-\eps)r} 
\le \frac{2}{(\eps r)^2}
\]
and
\[
\p{E_2+\ldots+E_{r} \ge (1+\eps)r} 
\le \frac{2}{(\eps r)^2},
\]
and hence 
\[
\p{A(\eps,r)^c} \le \frac{2}{r^2} + \frac{4}{(\eps r)^2} \leq \frac{6}{(\eps r)^2} \,.
\]
Moreover, $A(\eps,r)$ is independent of $E_1$, so $\Cprob{A(\eps,r)}{E_1}=\p{A(\eps,r)}$. 

With the preceding bounds at hand, we have enough information to control \eqref{eq:first_mt} and \eqref{eq:second_mt}. Writing $A=A(\eps,r)$, we have 
\begin{align}\label{eq:1st_mt_lb}
\probC{r \in \cU_1}{E_1} & = 
\Cexp{R\I{\pi_r \le 1}}{E_1}\nonumber\\
& 
\ge \Cexp{R\I{\pi_r \le 1}\I{A}}{E_1}\nonumber\\
& = \Cexp{R\I{A}}{E_1} \nonumber\\
& = \Cexp{R}{A,E_1}\Cprob{A}{E_1} \nonumber\\
& \ge \frac{1}{3} \frac{E_1^{1/2}}{(E_1+(1+\eps)r)^{1/2}}\pran{1-\frac{6}{(\eps r)^2}} \,,
\end{align}
and
\begin{align}\label{eq:1st_mt_ub}
\probC{r \in \cU_1}{E_1}
& \le 
\Cexp{R\I{\pi_r \le 1}}{A,E_1} + \Cprob{A^c}{E_1}\nonumber\\
& \le \Cexp{R}{A,E_1} + \frac{6}{(\eps r)^2} \nonumber\\
& \le \frac{1}{3} \frac{E_1^{1/2}}{((1-\eps) r)^{1/2}} + \frac{6}{(\eps r)^2}\, .
\end{align}
We similarly have 
\begin{align}\label{eq:2nd_mt_ub}
\probC{r-1\in \cU_1,r \in \cU_1}{E_1}& = 
\Cexp{R'\I{\pi_{r-1},\pi_r \le 1}}{E_1}\nonumber\\
& 
\le 
\Cexp{R'\I{\pi_{r-1},\pi_r \le 1}}{E_1,A} + \Cprob{A^c}{E_1}\nonumber\\
& \le \Cexp{R'}{A,E_1} + \frac{6}{(\eps r)^2} \nonumber\\
& \le \frac{1}{9}\frac{E_1}{(1-\eps)r} + \frac{6}{(\eps r)^2}\, .
\end{align}
We use the above bounds in the following formulas, which are immediate consequences of Lemma~\ref{lem:U1_exchange} (using the fact that conditioning on $\pi_0$ and on $E_1$ is equivalent):
\begin{equation}\label{eq:exp_formula}
\Cexp{|\cU_1|}{E_1} 
= r\probC{r \in \cU_1}{E_1}
\end{equation}
and 
\begin{align}\label{eq:var_formula}
& \va\left\{|\cU_1|\right|\left.E_1\right\} \nonumber \\
& 
= r(r-1) 
\left(
\Cprob{r-1,r \in \cU_1}{E_1}
-\Cprob{r \in \cU_1}{E_1}^2
\right) \nonumber\\
&
+ r
\left(\Cprob{r \in \cU_1}{E_1}-\Cprob{r \in \cU_1}{E_1}^2\right)\,.
\end{align}
We now temporarily work on the event that $E_1 \in (\eps, \eps^{-1})$. On this event, there is $r_0=r_0(\eps) > 0$ such that for all $r \geq r_0$, 
\begin{align*}
\frac{1}{3} \frac{E_1^{1/2}}{(E_1+(1+\eps)r)^{1/2}} \left(1 -\frac{6}{(\eps r)^2} \right)
&\geq 
\frac{1}{3} \frac{E_1^{1/2}}{((1+2\eps)r)^{1/2}} \,, \\
\frac{1}{3} \frac{E_1^{1/2}}{((1-\eps) r)^{1/2}} + \frac{6}{(\eps r)^2}
&\leq
\frac{1}{3} \frac{E_1^{1/2}}{((1-2\eps) r)^{1/2}} \,, \text{and} \\ 
\frac{1}{9}\frac{E_1}{(1-\eps)r} + \frac{6}{(\eps r)^2}
&\leq
\frac{1}{9}\frac{E_1}{(1-2\eps)r} \,.
\end{align*} 
Using the first of these bounds together with \eqref{eq:1st_mt_lb} in \eqref{eq:exp_formula} gives 
that on the event $\{E_1 \in (\eps,\eps^{-1})\}$,
\begin{equation}\label{eq:exp_lb}
\Cexp{|\cU_1|}{E_1} 
\ge \frac{r^{1/2}}{3}\frac{E_1^{1/2}}{(1+2\eps)^{1/2}}\, ,
\end{equation}
and using all three bounds 
together with \eqref{eq:1st_mt_lb},\eqref{eq:1st_mt_ub} and \eqref{eq:2nd_mt_ub} 
in \eqref{eq:var_formula}, we obtain 
that on the event $\{E_1\in (\eps,\eps^{-1})\}$,
\begin{align*}
& \va
\left\{|\cU_1|\right|\left.E_1\right\} \\
& \leq
r(r-1)
\left(
\frac{1}{9}\frac{E_1}{(1-2\eps)r} - \frac{1}{9}\frac{E_1}{(1+2\eps)r}\right) + 
r\left(
\frac{1}{3} \frac{E_1^{1/2}}{((1-2\eps)r)^{1/2}}
\right)\\
& = (r-1) \frac{E_1}{9} \frac{4\eps}{1-4\eps^2} + r^{1/2} \frac{E_1^{1/2}}{3} \frac{1}{(1-2\eps)^{1/2}}\, \\
& < 5 \eps \Cexp{|\cU_1|}{E_1}^2\, ,
\end{align*}
the final bound holding by~\eqref{eq:exp_lb} for $\eps$ sufficiently small ($\eps < 1/8$ is enough), and still provided that and $r \geq r_0$. 

The lower bound in \eqref{eq:exp_lb} is at least $(r E_1)^{1/2}/4$ provided $\eps$ is small enough, so it now follows by the conditional Chebyshev inequality that 
\begin{align*}
& \Cprob{|\cU_1| \le \frac{(r E_1)^{1/2}}{8},E_1 \in (\eps,\eps^{-1})}{E_1}\\
& \le 
\frac{\va
\left\{|\cU_1|\right|\left.E_1\right\} \I{E_1 \in (\eps,\eps^{-1})}}{(\Cexp{|\cU_1|}{E_1}/2)^2} \\
& \le 
20\eps \I{E_1 \in (\eps,\eps^{-1})} \,.
\end{align*}
For $\eps$ small enough, if $E_1 \ge \eps$ then $E_1^{1/2}/8 > 2\eps$,
so the preceding bound implies that, unconditionally, 
\begin{align*}
& \p{|\cU_1| \le \eps r^{1/2}}\\
& 
=
\E{\Cprob{|\cU_1| \le \eps r^{1/2}}{E_1}}\\
& \le 
\E{\Cprob{|\cU_1| \le \frac{(r E_1)^{1/2}}{8}}{E_1}\I{E_1 \in \eps,\eps^{-1}}}
+ \p{E_1 \not\in(\eps,\eps^{-1})}\\
& \le 22\eps\, ,
\end{align*}
the last bound holding since $\p{E_1 \not\in\eps,\eps^{-1}} < 2\eps$ for $\eps$ small. This establishes \eqref{eq:toprove} and completes the proof. 
\end{proof}

\subsubsection{Proof of Proposition~\ref{prop:pacb}}
\label{sec:s>0}
Having already proved Proposition~\ref{prop:path-breaking fixed r}, to complete the proof of Proposition~\ref{prop:pacb} it remains to handle the cases when $s > 0$. So fix $\eps > 0$ and integers $s > 0$ and $r > 0$, and let $G_q$ and $F_q=F_q(r,\be)$ be as in the statement of Proposition~\ref{prop:pacb}. 
Like in the case $s=0$, it suffices to prove that if $r$ is sufficiently large as a function of $\eps$ and $s$ then for all $q$ sufficiently large, if $X$ and $Y$ are independent, uniformly random elements of $[q]$, independent of $G_q$ and of the ordering $\be$, then 
\[
\p{X \lr{F_q} Y} < \eps\, .
\]
To accomplish this, we decompose $G_q$ into a collection of trees to which we can apply the result from the $s=0$ case, Proposition~\ref{prop:path-breaking fixed r}. We next turn to defining the necessary decomposition. The definitions of the next four paragraphs are illustrated in Figure~\ref{fig:kernel_decomp}.

Let $\core(G_q)$ be the maximum induced subgraph of $G_q$ with minimum degree $2$; equivalently, this is the subgraph of $G_q$ induced by the set of vertices which lie on cycles of $G_q$. For $v \in [q]$ let $c(v)$ be the (unique) closest vertex of $\core(G_q)$ to $v$ in $G_q$. In particular, if $v \in v(\core(G_q))$ then $c(v)=v$. 

If $s \ge 2$ then the {\em kernel} of $G_q$, denoted $K(G_q)$, is  the multigraph obtained from $\core(G_q)$ by contracting each path whose endpoints have degree at least three in $\core(G_q)$ and whose internal vertices have degree two in $\core(G_q)$ into a single edge. 
For each vertex $v$ of $G_q$, we define its ``attachment location on $K(G_q)$'', denoted $\kappa(v)$, as follows. 
For each edge $e$ of $K(G_q)$, if $c(v)$ is an internal vertex of the path which was contracted to make $e$, then set $\kappa(v)=e$. Otherwise, if $c(v)$ is a vertex $w$ of $K(G_q)$ then set $\kappa(v)=w$. 

If $s=1$ then $\core(G_q)$ is a cycle. It is still useful for us to define the kernel in this case, but the definition is slightly different (and slightly non-standard). To define it, we first augment the core by adding all vertices of the path from $q$ to $c(q)$; we write $\core^+(G_q)$ for the subgraph of $G_q$ induced by this path together with $\core(G_q)$. We then define the kernel $K(G_q)$ to be the multigraph obtained from $G_q$ by contracting each maximal path or cycle of $\core^+(G_q)$ whose endpoints lie in $\{q \cup c(q)\}$ to form a single edge. If $q \ne c(q)$ then this creates a ``lollipop'' consisting of a loop edge at $c(q)$ and a single edge from $c(q)$ to $q$; if $q=c(q)$ then the result is simply a loop edge at $c(q)$. 

Provided that $s \ge 1$, so that the kernel is defined, for $a \in v(K(G_q)) \cup e(K(G_q))$ we now set $V_q(a) = \{v \in [q]: \kappa(v)=a\}$. 
Then the set 
\begin{equation}\label{eq:partition}
\bV_q=
\{ 
V_q(a),a \in v(K(G_q)) \cup e(K(G_q)) 
\} 
\end{equation}
is a partition of $v(G_q)=[q]$. For each $a \in v(K(G_q)) \cup e(K(G_q))$, we let $T_q(a)$ be the subgraph of $G_q$ spanned by $V_q(a)$.
Also, for $e=xy \in e(K(G_q))$, we write $Z(e,x)$ (resp.\ $Z(e,y)$) for the unique vertex of $T_q(e)$ incident to $x$ (resp.\ to $y$). 

\begin{figure}[htb]
\includegraphics[width=\textwidth]{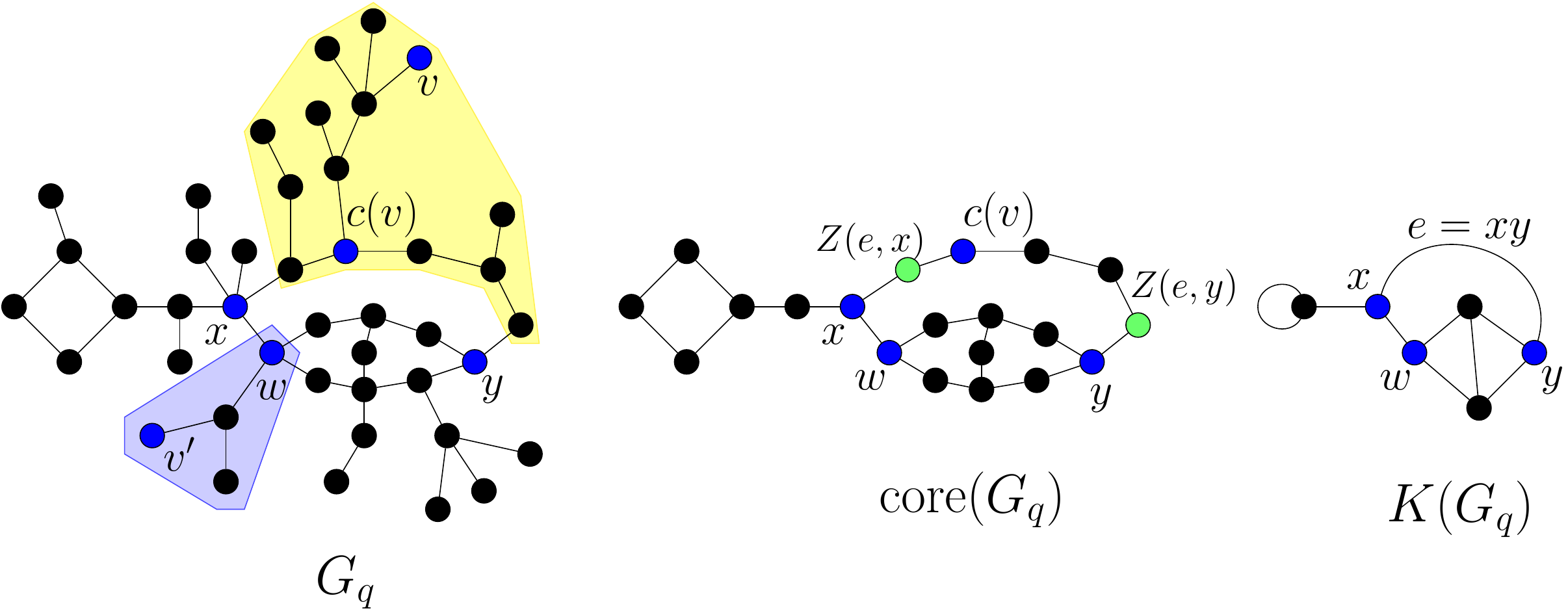}
\caption{Left: An instantiation of graph $G_q$; here $q=43$ and $s=3$. Center: the graph $\core(G_q)$. Right: the kernel $K(G_q)$. In the graph $G_q$, the vertex $v$ has $\kappa(v)=e=xy$ since $c(v)$ lies on the path of $\core(G_q)$ which is contracted to form $e$. The vertex $v'$ has $\kappa(v')=w$ since $c(v)=w$ is a vertex of $K(G_q)$. The trees $T_q(e)$ and $T_q(w)$ are highlighted in yellow and in blue, respectively. In the center, the vertices $Z(e,x)$ and $Z(e,y)$ are green.}
\label{fig:kernel_decomp}
\end{figure}

By the definition of the core, $T_q(a)$ is necessarily a tree. By the symmetry of the model, conditionally given the partition $\bV_q$ in \eqref{eq:partition}, the trees $(T_q(a),a \in v(K(G_q)) \cup e(K(G_q)))$ are independent and each is a uniformly random tree on its vertex set. Moreover, also by symmetry, for each $e \in e(K(G_q))$, conditionally given both $\bV_q$ and the tree $T_q(e)$, the vertices $Z(e,u)$ and $Z(e,v)$ are independent uniformly random elements of $V_q(e)$. 

The next proposition describes the asymptotic structure of the partition of mass in \eqref{eq:partition}. 
For each positive integer $k$, let 
\[
\Delta_k = \{(x_1,\ldots,x_k)\in (0,1)^k: x_1+\ldots+x_k=1\} 
\]
denote the $(k-1)$-dimensional simplex. Then 
for $(\alpha_1,\ldots,\alpha_k)\in \Delta_k$, the Dirichlet$(\alpha_1,\ldots,\alpha_k)$ distribution on $\Delta_k$ has density 
\[
\frac{\Gamma(\alpha_1+\ldots+\alpha_k)}{\Gamma(\alpha_1)\cdot\ldots\cdot \Gamma(\alpha_k)}\prod_{j=1}^k x_j^{\alpha_j-1} 
\]
with respect to $(k-1)$-dimensional Lebesgue measure on $\Delta_k$. 
\begin{prop}[\cite{AddBroGol12} Theorem 22, \cite{abbr2010} Theorem 6 (c)]\label{prop:dirichlet_limit} 
Fix $s \ge 1$ and let $G_q$ be uniformly distributed over the set of connected graphs with vertex set $[q]$ and surplus~$s$. 
Then as $q \to \infty$, the vector 
\[
(q^{-1}V_q(e),e \in e(K(G_q))) 
\] 
converges in distribution to a Dirichlet$(1/2,\ldots,1/2)$ random vector of length $k=2s-1+\I{s=1}$. 
\end{prop}
In the vector in  Proposition~\ref{prop:dirichlet_limit} we may take the edges of $K(G_q)$ to be ordered lexicographically, say, but the precise ordering rule does not play an important role in this paper. 
 
Recall that $X$ and $Y$ are independent, uniformly random elements of $[q]$, independent of $G_q$ and of the ordering $\be$. Then 
\begin{align*}
\p{X \lr{F_q} Y} 
& = \E{\probC{E_q}{F_q}} \\
& \ge \E{\frac{\max(|C|:C\mbox{ is a component of }F_q)^2}{q^2}}\\
& \ge 
\frac{(\E{\max(|C|:C\mbox{ is a component of }F_q)})^2}{q^2}\, ,
\end{align*}
so to accomplish our goal it suffices to show that $\p{X \lr{F_q} Y} \le \eps^2$ if $r$ is large enough. Since $\eps > 0$ was arbitrary, we may as well just show that $\p{X \lr{F_q} Y}<\eps$ for $r$ large. 

Let $A$ be the event that $\kappa(X) \in e(K(G_q))$ and $\kappa(Y) \in e(K(G_q))$, let $B= A \cap \{\kappa(X)=\kappa(Y)\}$ and let $C= A \cap \{\kappa(X)\ne\kappa(Y)\}$. 
By Proposition~\ref{prop:dirichlet_limit}, 
$q^{-1}\sum_{e \in e(K(G_q))} V_q(e) \to 1$ in probability, which implies that
$\p{A}\to 1$ as $q \to \infty$. 
By the same proposition, the limits 
\[
\lim_{q \to \infty} \p{B} =p =1-\lim_{q \to \infty} \p{C}
\]
both exist, and the value $p$ lies strictly between $0$ and $1$. 

Now let $\delta > 0$ be small enough that for $q$ large, 
\[
\p{\min(|V_q(e)|,e \in e(K(G_q))) < \delta q} < \min(p,1-p)\eps/7\, ;
\]
such a value $\delta$ exists by Proposition~\ref{prop:dirichlet_limit}. 
Then by Bayes' formula, for $q$ sufficiently large, 
\begin{equation}\label{eq:edge_trees_big}
\Cprob{\min(|V_q(e)|,e \in e(K(G_q))) < \delta q}{B} < \eps/6
\end{equation}
and
\begin{equation}\label{eq:edge_trees_big_2}
\Cprob{\min(|V_q(e)|,e \in e(K(G_q))) < \delta q}{C} < \eps/6.
\end{equation}

If $\kappa(X)=e=uv$, then any path from $X$ to $Y$ in $F_q$ must either lie within $T_q(e)$ or else must pass through one of $Z(e,u)$ or $Z(e,v)$. It follows that 
\begin{align*}
& \probC{X \lr{F_q} Y}{\bV_q,B} \\
& 
\le 
\probC{X \lr{T_q(e) \cap F_q} Y}{\bV_q,B} 
+ 
\probC{X \lr{T_q(e) \cap F_q} Z(e,u)}{\bV_q,B}
\\
& + \probC{X \lr{T_q(e) \cap F_q} Z(e,v)}{\bV_q,B} \\
& = 
3\probC{X \lr{T_q(e) \cap F_q} Y}{\bV_q,B} 
\, ,
\end{align*}
where for the final equality we have used that 
where we have used that conditionally given $B$ and $\bV_q$, the vertices $Z(e,u)$ and $Z(e,v)$ and $Y$ are all uniformly random elements of $V_q(e)$ independent of $X$, and where we write $T_q(e)\cap F_q$ to mean the subgraph of $T_q(e)$ with edge set $e(T_q(e))\cap e(F_q)$. 

Now note that the path-and-cycle-breaking process on $F_q$, when restricted to $T_q(e)$, removes a superset of the edges that would be removed by running the path-breaking process on $T_q(e)$ with the induced edge ordering. (This holds since removing an edge $e'$ of $T_q(e)$ may separate a pair of elements of $[r]$ one or both of which lie outside of $V_q(e)$; in this case, the edge $e'$ is removed in the path-and-cycle-breaking process. However, $e'$ may not be removed in the path-breaking process, if $e'$ does not separate a pair of elements of $[r]\cap V_q(e)$.) In other words, writing $F_q(e)'$ for the forest obtained by running the path-breaking process on $T_q(e)$ with starting set $[r]\cap V_q(e)$ and edge ordering given by the restriction of $\be$ to $e(T_q(e))$, then $T_q(e)\cap F_q$ is a sub-forest of $F_q(e)'$. It follows that, writing $e=\kappa(X)$, which is also equal to $\kappa(Y)$ when $B$ occurs, we have 
\begin{align}\label{eq:fcp}
\probC{X \lr{F_q} Y}{\bV_q,B} 
& 
\le 
3\probC{X \lr{T_q(e) \cap F_q} Y}{\bV_q,B}
\nonumber\\
& \le
3\probC{X \lr{F_q(e)'} Y}{\bV_q,B}\, .
\end{align}

Now note that if $G$ is a fixed graph with vertex set $[q]$ whose largest connected component has $c$ vertices, and $X$ and $Y$ are independent uniformly random elements of $[q]$, then 
$\p{X \lr{G} Y}\le \tfrac{c}{q}$. 
If $G$ is instead random, then this bound and the tower law give that 
\[
\p{X \lr{G} Y}
\le q^{-1}\E{\max(|C|:C\mbox{ is a component of }F_q)}.
\]
It thus follows from Proposition~\ref{prop:path-breaking fixed r} 
there is $q_0$ such that if $|V_q(e)| \ge q_0$ then the conditional probability on the right of \eqref{eq:fcp} is less than $\eps/12$, so we have 
\[
\probC{X \lr{F_q} Y}{\bV_q,B} < 
3(\eps/12)\I{V_q(\kappa(X))\ge q_0}
+ 3\I{V_q(\kappa(X))< q_0}, 
\]
which together with \eqref{eq:edge_trees_big} yields that for $q$ large enough (and in particular large enough that $\delta q > q_0$), 
\begin{align*}
\probC{X \lr{F_q} Y}{B}
& = \E{\probC{X \lr{F_q} Y}{\bV_q,B}~|~B}\\
&  
(\eps/4) \probC{V_q(\kappa(X))\ge q_0}{B}
+ 3\probC{V_q(\kappa(X))< q_0}{B} \\
& < 3\eps/4 \, .
\end{align*}

A nearly identical proof, but using \eqref{eq:edge_trees_big_2} in place of \eqref{eq:edge_trees_big}, shows that $\probC{X \lr{F_q} Y}{C} < 3\eps/4$ for all $q$ sufficiently large. (In fact, in this case we could obtain a slightly better bound, since when $C$ occurs, in order for $X$ and $Y$ to be connected in $F_q$ there must be a path from $X$ to $Z(e,u)$ or $Z(e,v)$ in $T_q(e)\cap F_q$; the term $X \lr{T_q(e)\cap F_q} Y$ does not appear.) Since if $A$ occurs then either $B$ or $C$ must occur, it follows that 
\begin{align*}
\p{X \lr{F_q} Y}  
& \le \p{X \lr{F_q} Y,B} + \p{X \lr{F_q} Y,C} + \p{A^c} \\
& \le \frac{3\eps}{4} + \p{A} < \eps
\end{align*}
the last two inequalities holding for all $q$ sufficiently large. This completes the proof of Proposition~\ref{prop:pacb} in the case $s >0$. 

\section{Conclusion}\label{sec:conc}

In addition to the conjectures raised directly after the statement of Theorem~\ref{thm:main}, there are numerous avenues for future research suggested by the current work.

First, we expect that a version of the dichotomy established in Theorem~\ref{thm:main} should hold for other high-dimensional random graphs, at least those with sufficient symmetry. For example, we expect that the same theorem should hold if $K_n$ is replaced by a uniformly random $d$-regular graph (for $d \ge 3$), or by the nearest-neighbour hypercube $\{0,1\}^N$ with $2^N \asymp n$. A version of the theorem may well also hold in high-dimensional lattice tori (i.e.\ with $K_n$ replaced by $(\Z/m\Z)^d$, where $m^d \asymp n$, with $d$ fixed and large). However, in Euclidean settings there is less symmetry; nearby sources are in more direct competition than far-off sources, and it is not clear to us how substantially this will affect the behaviour of the multi-source invasion process.

The behaviour in low-dimensional settings is of course also interesting. It's possible that  enough is known about two-dimensional critical percolation (at least on the triangular lattice \cite{MR3790067}) to be able to make some progress on the structure of multi-source invasion percolation.

Our results suggest the following behaviour for multi-source invasion percolation on large conditioned critical Bienaym\'e trees\footnote{We follow the terminological suggestion of \cite{ab21universal}, using the term ``Bienaym\'e trees'' rather than ``Galton-Watson trees'' for the family trees of branching processes.} with finite variance offspring distribution. For such trees, invasion percolation from boundedly many sources (i.e.\ with $k(n)=k$ fixed) will result in all components having macroscopic sizes which are random to first order; on the other hand, invasion percolation from unboundedly many sources (i.e.\ with $k(n) \to \infty$) will with high probability result in all components of sublinear size. This can likely be proved in detail using weak convergence arguments similar to those used to study the ``Markov chainsaw'' in \cite{MR3262504}. In both cases, it would be would be of interest to understand the distribution of component sizes; in the case of unboundedly many sources, the precise behaviour of the size of the largest connected component is unclear to us, and may depend more sensitively on the offspring distribution, at least if $k(n) \to \infty$ sufficiently quickly. 

It is less clear to us what should happen for conditioned critical Bienaym\'e trees with infinite variance (e.g.\ stable trees). In this setting, the presence of hubs -- nodes with very large degree - could play an important role in the dynamics of the invasion process. 

For other models of random trees and networks (e.g.\ preferential attachment networks, inhomogeneous random graphs, or networks with community structure, or any sort of directed models), the subject is wide open.

\addtocontents{toc}{\SkipTocEntry} 
\section{Acknowledgements}
The authors thank Ross Kang for pointing out the paper \cite{logan18variant}. 
During the preparation of this research, LAB was supported by an NSERC Discovery Grant and a Simons Fellowship in Mathematics.

\small 

\bibliographystyle{plainnat}
\bibliography{gnpf}

\begin{thebibliography}{33}
\providecommand{\natexlab}[1]{#1}
\providecommand{\url}[1]{\texttt{#1}}
\expandafter\ifx\csname urlstyle\endcsname\relax
  \providecommand{\doi}[1]{doi: #1}\else
  \providecommand{\doi}{doi: \begingroup \urlstyle{rm}\Url}\fi

\bibitem[Addario-Berry(2013)]{ablocal}
Louigi Addario-Berry.
\newblock The local weak limit of the minimum spanning tree of the complete
  graph.
\newblock arXiv:1301.1667 [math.PR], January 2013.

\bibitem[Addario-Berry and Sen(2021)]{MR4288328}
Louigi Addario-Berry and Sanchayan Sen.
\newblock Geometry of the minimal spanning tree of a random 3-regular graph.
\newblock \emph{Probab. Theory Related Fields}, 180\penalty0 (3-4):\penalty0
  553--620, 2021.
\newblock ISSN 0178-8051.
\newblock \doi{10.1007/s00440-021-01071-3}.
\newblock URL \url{https://doi.org/10.1007/s00440-021-01071-3}.

\bibitem[Addario-Berry et~al.((2012))Addario-Berry, Broutin, and
  Goldschmidt]{AddBroGol12}
Louigi Addario-Berry, Nicolas Broutin, and Christina Goldschmidt.
\newblock The continuum limit of critical random graphs.
\newblock \emph{Probab. Theory Related Fields}, {152}\penalty0 (3-4):\penalty0
  367--406, (2012).
\newblock ISSN 0178-8051.
\newblock \doi{10.1007/s00440-010-0325-4}.
\newblock URL
  \url{http://dx.doi.org.dianus.libr.tue.nl/10.1007/s00440-010-0325-4}.

\bibitem[Addario-Berry et~al.(2010)Addario-Berry, Broutin, and
  Goldschmidt]{abbr2010}
Louigi Addario-Berry, Nicolas Broutin, and Christina Goldschmidt.
\newblock {Critical Random Graphs: Limiting Constructions and Distributional
  Properties}.
\newblock \emph{Electronic Journal of Probability}, 15\penalty0
  (none):\penalty0 741 -- 775, 2010.
\newblock \doi{10.1214/EJP.v15-772}.
\newblock URL \url{https://doi.org/10.1214/EJP.v15-772}.

\bibitem[Addario-Berry et~al.(2012)Addario-Berry, Griffiths, and
  Kang]{MR2977982}
Louigi Addario-Berry, Simon Griffiths, and Ross~J. Kang.
\newblock Invasion percolation on the {P}oisson-weighted infinite tree.
\newblock \emph{Ann. Appl. Probab.}, 22\penalty0 (3):\penalty0 931--970, 2012.
\newblock ISSN 1050-5164.
\newblock \doi{10.1214/11-AAP761}.
\newblock URL \url{https://doi.org/10.1214/11-AAP761}.

\bibitem[Addario-Berry et~al.(2014)Addario-Berry, Broutin, and
  Holmgren]{MR3262504}
Louigi Addario-Berry, Nicolas Broutin, and Cecilia Holmgren.
\newblock Cutting down trees with a {M}arkov chainsaw.
\newblock \emph{Ann. Appl. Probab.}, 24\penalty0 (6):\penalty0 2297--2339,
  2014.
\newblock ISSN 1050-5164.
\newblock \doi{10.1214/13-AAP978}.
\newblock URL \url{https://doi.org/10.1214/13-AAP978}.

\bibitem[Addario-Berry et~al.(2017)Addario-Berry, Broutin, Goldschmidt, and
  Miermont]{MR3706739}
Louigi Addario-Berry, Nicolas Broutin, Christina Goldschmidt, and Gr\'{e}gory
  Miermont.
\newblock The scaling limit of the minimum spanning tree of the complete graph.
\newblock \emph{Ann. Probab.}, 45\penalty0 (5):\penalty0 3075--3144, 2017.
\newblock ISSN 0091-1798.
\newblock \doi{10.1214/16-AOP1132}.
\newblock URL \url{https://doi.org/10.1214/16-AOP1132}.

\bibitem[Addario-Berry et~al.(2021)Addario-Berry, Brandenberger, Hamdan, and
  Kerriou]{ab21universal}
Louigi Addario-Berry, Anna Brandenberger, Jad Hamdan, and C\'eline Kerriou.
\newblock Universal height and width bounds for random trees.
\newblock arXiv:2105.03195 [math.PR], May 2021.

\bibitem[Aldous(1991)]{aldous91continuum1}
David Aldous.
\newblock The continuum random tree. {I}.
\newblock \emph{Ann. Probab.}, 19\penalty0 (1):\penalty0 1--28, 1991.
\newblock ISSN 0091-1798.
\newblock URL
  \url{http://links.jstor.org/sici?sici=0091-1798(199101)19:1<1:TCRTI>2.0.CO;2-B&origin=MSN}.

\bibitem[Aldous(1997)]{aldous97brownian}
David Aldous.
\newblock Brownian excursions, critical random graphs and the multiplicative
  coalescent.
\newblock \emph{Ann. Probab.}, 25\penalty0 (2):\penalty0 812--854, 1997.
\newblock ISSN 0091-1798.
\newblock \doi{10.1214/aop/1024404421}.
\newblock URL \url{https://doi.org/10.1214/aop/1024404421}.

\bibitem[Aldous and Steele(2004)]{MR2023650}
David Aldous and J.~Michael Steele.
\newblock The objective method: probabilistic combinatorial optimization and
  local weak convergence.
\newblock In \emph{Probability on discrete structures}, volume 110 of
  \emph{Encyclopaedia Math. Sci.}, pages 1--72. Springer, Berlin, 2004.
\newblock \doi{10.1007/978-3-662-09444-0\_1}.
\newblock URL \url{https://doi.org/10.1007/978-3-662-09444-0_1}.

\bibitem[Aldous(1985)]{MR883646}
David~J. Aldous.
\newblock Exchangeability and related topics.
\newblock In \emph{\'{E}cole d'\'{e}t\'{e} de probabilit\'{e}s de
  {S}aint-{F}lour, {XIII}---1983}, volume 1117 of \emph{Lecture Notes in
  Math.}, pages 1--198. Springer, Berlin, 1985.
\newblock \doi{10.1007/BFb0099421}.
\newblock URL \url{https://doi.org/10.1007/BFb0099421}.

\bibitem[Angel et~al.(2008)Angel, Goodman, den Hollander, and Slade]{MR2393988}
Omer Angel, Jesse Goodman, Frank den Hollander, and Gordon Slade.
\newblock Invasion percolation on regular trees.
\newblock \emph{Ann. Probab.}, 36\penalty0 (2):\penalty0 420--466, 2008.
\newblock ISSN 0091-1798.
\newblock \doi{10.1214/07-AOP346}.
\newblock URL \url{https://doi.org/10.1214/07-AOP346}.

\bibitem[Angel et~al.(2013)Angel, Goodman, and Merle]{MR3059198}
Omer Angel, Jesse Goodman, and Mathieu Merle.
\newblock Scaling limit of the invasion percolation cluster on a regular tree.
\newblock \emph{Ann. Probab.}, 41\penalty0 (1):\penalty0 229--261, 2013.
\newblock ISSN 0091-1798.
\newblock \doi{10.1214/11-AOP731}.
\newblock URL \url{https://doi.org/10.1214/11-AOP731}.

\bibitem[Bhamidi and Sen(2020)]{MR4084187}
Shankar Bhamidi and Sanchayan Sen.
\newblock Geometry of the vacant set left by random walk on random graphs,
  {W}right's constants, and critical random graphs with prescribed degrees.
\newblock \emph{Random Structures Algorithms}, 56\penalty0 (3):\penalty0
  676--721, 2020.
\newblock ISSN 1042-9832.
\newblock \doi{10.1002/rsa.20880}.
\newblock URL \url{https://doi.org/10.1002/rsa.20880}.

\bibitem[Bhamidi et~al.(2018)Bhamidi, van~der Hofstad, and Sen]{MR3748328}
Shankar Bhamidi, Remco van~der Hofstad, and Sanchayan Sen.
\newblock The multiplicative coalescent, inhomogeneous continuum random trees,
  and new universality classes for critical random graphs.
\newblock \emph{Probab. Theory Related Fields}, 170\penalty0 (1-2):\penalty0
  387--474, 2018.
\newblock ISSN 0178-8051.
\newblock \doi{10.1007/s00440-017-0760-6}.
\newblock URL \url{https://doi.org/10.1007/s00440-017-0760-6}.

\bibitem[Chandler et~al.(1982)Chandler, Koplik, Lerman, and
  Willemsen]{chandler_koplik_lerman_willemsen_1982}
Richard Chandler, Joel Koplik, Kenneth Lerman, and Jorge~F. Willemsen.
\newblock Capillary displacement and percolation in porous media.
\newblock \emph{Journal of Fluid Mechanics}, 119:\penalty0 249--267, 1982.
\newblock \doi{10.1017/S0022112082001335}.

\bibitem[Damron and Sapozhnikov(2012)]{MR2962082}
Michael Damron and Art\"{e}m Sapozhnikov.
\newblock Limit theorems for 2{D} invasion percolation.
\newblock \emph{Ann. Probab.}, 40\penalty0 (3):\penalty0 893--920, 2012.
\newblock ISSN 0091-1798.
\newblock \doi{10.1214/10-AOP641}.
\newblock URL \url{https://doi.org/10.1214/10-AOP641}.

\bibitem[Garban et~al.(2018{\natexlab{a}})Garban, Pete, and Schramm]{MR3790067}
Christophe Garban, G\'{a}bor Pete, and Oded Schramm.
\newblock The scaling limits of near-critical and dynamical percolation.
\newblock \emph{J. Eur. Math. Soc. (JEMS)}, 20\penalty0 (5):\penalty0
  1195--1268, 2018{\natexlab{a}}.
\newblock ISSN 1435-9855.
\newblock \doi{10.4171/JEMS/786}.
\newblock URL \url{https://doi.org/10.4171/JEMS/786}.

\bibitem[Garban et~al.(2018{\natexlab{b}})Garban, Pete, and Schramm]{MR3857861}
Christophe Garban, G\'{a}bor Pete, and Oded Schramm.
\newblock The scaling limits of the minimal spanning tree and invasion
  percolation in the plane.
\newblock \emph{Ann. Probab.}, 46\penalty0 (6):\penalty0 3501--3557,
  2018{\natexlab{b}}.
\newblock ISSN 0091-1798.
\newblock \doi{10.1214/17-AOP1252}.
\newblock URL \url{https://doi.org/10.1214/17-AOP1252}.

\bibitem[Kruskal(1956)]{kruskal56mst}
Joseph~B. Kruskal, Jr.
\newblock On the shortest spanning subtree of a graph and the traveling
  salesman problem.
\newblock \emph{Proc. Amer. Math. Soc.}, 7:\penalty0 48--50, 1956.
\newblock ISSN 0002-9939.
\newblock \doi{10.2307/2033241}.
\newblock URL \url{https://doi.org/10.2307/2033241}.

\bibitem[\L~uczak(1990)]{MR1099794}
Tomasz \L~uczak.
\newblock Component behavior near the critical point of the random graph
  process.
\newblock \emph{Random Structures Algorithms}, 1\penalty0 (3):\penalty0
  287--310, 1990.
\newblock ISSN 1042-9832.
\newblock \doi{10.1002/rsa.3240010305}.
\newblock URL \url{https://doi.org/10.1002/rsa.3240010305}.

\bibitem[Logan et~al.(2018)Logan, Molloy, and Pralat]{logan18variant}
Adam Logan, Mike Molloy, and Pawel Pralat.
\newblock A variant of the {E}rdos-{R}enyi random graph process.
\newblock arXiv:1806.10975 [math.CO], 2018.

\bibitem[McDiarmid et~al.(1997)McDiarmid, Johnson, and Stone]{MR1611522}
Colin McDiarmid, Theodore Johnson, and Harold~S. Stone.
\newblock On finding a minimum spanning tree in a network with random weights.
\newblock \emph{Random Structures Algorithms}, 10\penalty0 (1-2):\penalty0
  187--204, 1997.
\newblock ISSN 1042-9832.

\bibitem[Michelen et~al.(2019)Michelen, Pemantle, and Rosenberg]{MR3940761}
Marcus Michelen, Robin Pemantle, and Josh Rosenberg.
\newblock Invasion percolation on {G}alton-{W}atson trees.
\newblock \emph{Electron. J. Probab.}, 24:\penalty0 Paper No. 31, 35, 2019.
\newblock \doi{10.1214/19-EJP281}.
\newblock URL \url{https://doi.org/10.1214/19-EJP281}.

\bibitem[Newman and Stein(1995)]{MR1340039}
C.~M. Newman and D.~L. Stein.
\newblock Random walk in a strongly inhomogeneous environment and invasion
  percolation.
\newblock \emph{Ann. Inst. H. Poincar\'{e} Probab. Statist.}, 31\penalty0
  (1):\penalty0 249--261, 1995.
\newblock ISSN 0246-0203.
\newblock URL \url{http://www.numdam.org/item?id=AIHPB_1995__31_1_249_0}.

\bibitem[Newman and Stein(1996)]{MR1372437}
C.~M. Newman and D.~L. Stein.
\newblock Ground-state structure in a highly disordered spin-glass model.
\newblock \emph{J. Statist. Phys.}, 82\penalty0 (3-4):\penalty0 1113--1132,
  1996.
\newblock ISSN 0022-4715.
\newblock \doi{10.1007/BF02179805}.
\newblock URL \url{https://doi.org/10.1007/BF02179805}.

\bibitem[Nickel and Wilkinson(1983)]{MR708864}
Bernie Nickel and David Wilkinson.
\newblock Invasion percolation on the {C}ayley tree: exact solution of a
  modified percolation model.
\newblock \emph{Phys. Rev. Lett.}, 51\penalty0 (2):\penalty0 71--74, 1983.
\newblock ISSN 0031-9007.
\newblock \doi{10.1103/PhysRevLett.51.71}.
\newblock URL \url{https://doi.org/10.1103/PhysRevLett.51.71}.

\bibitem[Prim(1957)]{prim57shortest}
R.~C. Prim.
\newblock Shortest connection networks and some generalizations.
\newblock \emph{The Bell System Technical Journal}, 36\penalty0 (6):\penalty0
  1389--1401, 1957.
\newblock \doi{10.1002/j.1538-7305.1957.tb01515.x}.

\bibitem[Stark(1991)]{Stark:1991aa}
Colin~P. Stark.
\newblock An invasion percolation model of drainage network evolution.
\newblock \emph{Nature}, 352\penalty0 (6334):\penalty0 423--425, 1991.
\newblock \doi{10.1038/352423a0}.
\newblock URL \url{https://doi.org/10.1038/352423a0}.

\bibitem[van~den Berg et~al.(2007)van~den Berg, J\'{a}rai, and
  V\'{a}gv\"{o}lgyi]{MR2350578}
Jacob van~den Berg, Antal~A. J\'{a}rai, and B\'{a}lint V\'{a}gv\"{o}lgyi.
\newblock The size of a pond in 2{D} invasion percolation.
\newblock \emph{Electron. Comm. Probab.}, 12:\penalty0 411--420, 2007.
\newblock ISSN 1083-589X.
\newblock \doi{10.1214/ECP.v12-1327}.
\newblock URL \url{https://doi.org/10.1214/ECP.v12-1327}.

\bibitem[Wilkinson and Willemsen(1983)]{MR725616}
David Wilkinson and Jorge~F. Willemsen.
\newblock Invasion percolation: a new form of percolation theory.
\newblock \emph{J. Phys. A}, 16\penalty0 (14):\penalty0 3365--3376, 1983.
\newblock ISSN 0305-4470.
\newblock URL \url{http://stacks.iop.org/0305-4470/16/3365}.

\bibitem[Zhang(1995)]{MR1316507}
Yu~Zhang.
\newblock The fractal volume of the two-dimensional invasion percolation
  cluster.
\newblock \emph{Comm. Math. Phys.}, 167\penalty0 (2):\penalty0 237--254, 1995.
\newblock ISSN 0010-3616.
\newblock URL \url{http://projecteuclid.org/euclid.cmp/1104271992}.

\end{thebibliography}

\end{document}